\documentclass[12pt]{amsart}

\usepackage{amsmath,amssymb,amsfonts,graphics}
\usepackage[all]{xy}

\usepackage{color}
\usepackage{xcolor}
\usepackage{comment}
\usepackage{hyperref}
\newtheorem{thm}{Theorem}[section]
\newtheorem{cor}[thm]{Corollary}

\newtheorem{lem}[thm]{Lemma}
\newtheorem{prop}[thm]{Proposition}
\newtheorem{thmx}{Theorem}

\theoremstyle{definition}
\newtheorem{defn}[thm]{Definition}
\theoremstyle{remark}
\newtheorem{rem}[thm]{Remark}

\numberwithin{equation}{section}

\newcommand{\R}{\mathbb{R}}
\newcommand{\N}{\mathbb{N}}
\newcommand{\F}{\mathbb{F}}

\newcommand{\C}{\mathbb{C}}

\newcommand{\Aut}{\textup{Aut}}

\newcommand{\cpf}{\rtimes_{\textup{PF}_p} G}

\newcommand{\om}{\omega}

\newcommand{\supp}{\text{supp}}

\newcommand{\la}{\langle}
\newcommand{\ra}{\rangle}

\newcommand{\mc}{\mathcal}
\newcommand{\wt}{\widetilde}
\newcommand{\ov}{\overline}
\newcommand{\lin}{\operatorname{span}}
\newcommand{\md}{d}
\newcommand{\eps}{\varepsilon}

\newcommand{\PF}[3]{\ensuremath{\mathrm{PF}_{#1}^*\left(#2; #3\right)}}

\DeclareMathOperator{\LL}{L}

\title[$PF_p$-crossed product functors]{Crossed product functors associated
to $\ell^p$-pseudofunctions}

\author[Krajczok]{Jacek Krajczok}
\address{Institute of Mathematics of the Polish Academy of Sciences, {\'S}niadeckich 8, 00-656 Warsaw, Poland}
\email{jkrajczok@impan.pl}

\author[Samei]{Ebrahim Samei}
\address{Department of Mathematics and Statistics, University of Saskatchewan, Saskatoon, Saskatchewan, S7N 5E6, Canada}
\email{ebrahim.samei@usask.ca}

\author[Siebenand]{Timo Siebenand}
\address{}
\email{timosiebenand@outlook.de}

\author[Skalski]{Adam Skalski}
\address{Institute of Mathematics of the Polish Academy of Sciences, ul. {\'S}niadeckich 8, 00–656, Warszawa, Poland}
\email{a.skalski@impan.pl}

\begin{document}

\begin{abstract}
We show that the $\ell^p$-pseudofunctions, which were recently shown to lead to exotic completions of group $C^*$-algebras  by Wiersma and the second named author, can be used to construct well-behaved crossed product functors in the sense of Buss, Echterhoff and Willett. The construction proceeds via introducing certain  Banach algebras, related to operators acting on Hilbert valued $\ell^p$-spaces, which a priori depend on the choice of a Hilbert space representation of the underlying C*-algebra. We prove that, in fact, the resulting algebras are isomorphic (with the isomorphism constant depending only on $p$), and hence their enveloping C*-algebras are isometrically isomorphic. This, in particular, means that the construction genuinely generalises the one studied earlier in the group case. The tools we develop allow us to show that for certain non-amenable actions, the resulting crossed product completions must indeed be exotic.
\end{abstract}



\maketitle

\section{Introduction}

Discrete, or more generally, locally compact groups have been a source of interesting and important examples of operator algebras since the beginning of the theory of the latter in the 1940s. In particular, the constructions of the universal and the reduced C*-algebra of a discrete group $G$ serve as archetypical examples of `abstract' and `concrete' C*-algebras, and to the very day play a central position in the study of operator algebras and their interactions with the geometric group theory. Recent years have brought a lot of interest in the so-called `exotic' group C*-algebras, i.e.\ C*-completions of the group ring $\C[G]$ sitting naturally between the universal and the reduced one. One of the starting points of their study was the article \cite{BG}, where Brown and Guentner introduced the so-called $\ell^p$-completions of the group ring, which spurred a lot of further activity (see \cite{SW2} and references therein). Choosing a completion as above corresponds to fixing a certain family of (weak equivalence classes of) Hilbert space representations of $G$, and indeed most of the constructions of exotic group C*-algebras begin by setting up such a family. The work of the second named author and Wiersma in \cite{SW2}, inspired by their earlier article \cite{SW1}, took a different perspective. In the first step, they considered the shift representation of $G$ on $\ell^p(G)$, with fixed $p \in [1, \infty)$, leading in a natural way to a Banach $*$-algebra $\textup{PF}_p^*(G)$, and only in the second step the relevant C*-algebra appeared, as the enveloping C*-algebra of $\textup{PF}_p^*(G)$. One should note that in several cases $C^*(\textup{PF}_p^*(G))$ coincides with the Brown-Guentner algebra $C^*_{\ell^p}(G)$, and in fact it is still not known whether they are ever different.

On the other hand, soon after the work of Brown and Guentner, the article \cite{BGM}, motivated by looking for an appropriate formulation of the Baum-Connes conjecture with coefficients, initiated a systematic study of the `exotic' crossed products. Again, it had been known for a long time that given an action of $G$ on a C*-algebra $A$, one can form canonical universal/reduced crossed product algebras. The key insight of \cite{BGM}, later fleshed out primarily in \cite{BEW}, was that one should not look at potentially exotic crossed product completions separately for each C*-dynamical system, but rather view them simultaneously, as a prescription for finding C*-norms for all possible $G$-actions. This has led to a notion of a \emph{crossed product functor}, with several notable examples related to the Brown-Guentner construction, but also to the study of exotic coactions initiated in \cite{KLQNYJM}. The article \cite{BEW} and its successors, such as \cite{AEES}, introduced also several natural and important properties of crossed product functors, among which we would like to distinguish being a \emph{correspondence functor} (the property roughly means that the construction behaves well with respect to $G$-equivariant C*-correspondences, but admits also several equivalent reformulations).

The aim of our paper is to place the construction of the $C^*(\textup{PF}_p^*(G))$ in the framework of crossed product functors. To that end, similarly to the path chosen in \cite{SW2}, to any action of a discrete group $G$ on a $C^*$-algebra $A$, we associate first a Banach $*$-algebra $\textup{PF}^*_p(G;A)$, constructed via a Banach space representation of the relevant convolution algebra on the $\ell^p$-space with values in a Hilbert space. Although the construction demands a choice of a representation of $A$, we exploit the Kahane-Khintchine inequality and the Voiculescu's theorem to prove that different choices lead to isomorphic objects, with the isomorphism bound dependent only on $p$.  Then the desired C*-algebra arrives as the enveloping C*-algebra of $\textup{PF}^*_p(G;A)$. Let $\rtimes_{\textup{PF}_p} G$ denote the functor $C^* \circ \textup{PF}^*_p(G;\cdot)$. The following theorem summarizes some of the main results of our work (see Theorem \ref{thm:coincide}, Corollary \ref{cor:indepBanach}, Theorem \ref{T:C* envlope PF cross product}, and Theorem \ref{thm:projectionprop}).

\begin{thmx} \label{thmA}
	The functor $\rtimes_{\textup{PF}_p} G$
	from the category of $G$-C*-algebras to the category of C*-algebras is a crossed product correspondence functor. The C*-algebra $C^*(\textup{PF}^*_p(G;A))$ is independent from the choice of the representation for $A$. Moreover, $C^*(\textup{PF}^*_p(G;\C))$ is canonically isometrically isomorphic to the C*-algebra  $C^*(\textup{PF}^*_p(G))$ introduced by Wiersma and the second named author in \cite{SW2}.
\end{thmx}

One should note that this implies, in particular, that the functor $\rtimes_{\textup{PF}_p} G$ is indeed different from the Brown-Guentner functor $\rtimes_{BG_p} G$ (as the latter is not a correspondence functor). Furthermore, we show that if the action of $G$ on $A$ admits an invariant state, then the C*-algebra $C^*(\textup{PF}_p^*(G))$ embeds canonically into $C^*(\textup{PF}^*_p(G;A))$. This means that we can directly use several results of \cite{SW2} to show that the functor we consider here is, in general, genuinely exotic, by which we mean that none of the canonical surjections $A\rtimes_u G \to C^*(\textup{PF}^*_p(G;A))$ and $C^*(\textup{PF}^*_p(G;A)) \to A\rtimes_r G$ is isometric (see Theorem \ref{T:exotic C*-alg imply exotic cross product}). We also exhibit certain instances where we can directly prove, using the random walk and Furstenberg entropy techniques introduced in a similar context in [ASdLSW$_{1-2}$], that  $C^*(\textup{PF}^*_p(G;A))$ are exotic without assuming the existence of an invariant state (see Theorem \ref{T:nontrivial PF_p cross C* alg using entropy} and Remark \ref{R:nontrivial PF_p cross C* alg using entropy}).

The detailed plan of the paper is as follows: after this introduction, in Section \ref{sec:cpf} we recall the basic aspects of the theory of crossed product functors, mostly following \cite{BEW}. Section \ref{sec:Banach} is devoted to the study of Banach algebras $\textup{PF}^*_p(G;A)$ associated to a $G$-C*-algebra $A$ and $p \in (1, \infty)$. Here we show that the construction (up to a Banach algebra isomorphism) does not depend on the choice of the representation of $A$, and establish several functorial properties of the functor $\textup{PF}^*_p(G;\cdot)$. This sets the ground for Section \ref{sec:C*}, where we define our C*-algebraic crossed product functor   $\rtimes_{\textup{PF}_p} G$ and prove that it is indeed a correspondence functor.
Finally, in Section \ref{sec:examples}, we treat several examples where $\rtimes_{\textup{PF}_p} G$ leads to genuinely exotic crossed products.

Throughout $\otimes$ denotes the minimal C*-tensor product of C*-algebras. Hilbert space scalar products are linear on the right side. We will use $G$ to denote a discrete group.

\section{Crossed product functors and their properties} \label{sec:cpf}

We shall describe here briefly basic aspects of the theory of crossed product functors, as introduced by Buss, Echterhoff and Willett in \cite{BEW}. Note that the theory has been developed for general locally compact groups, but we shall stick to the discrete context.

Let $G$ be a discrete group. By a $G$-C*-algebra we understand a C*-algebra $A$ equipped with a homomorphism $\alpha\colon G \to \Aut(A)$. We shall denote by $A\rtimes_{\mathrm{alg}}G$ the \emph{algebraic crossed product}, so the space of finitely supported functions on $G$ with values in $A$ equipped with the $*$-algebra structure induced by the action $\alpha$; note that sometimes it will be useful to speak also about the corresponding `$\ell^1$-algebra', i.e.\ $\ell^1(G;A)$.
We shall often speak of the category of $G$-C*-algebras, with the morphisms given by  $G$-equivariant $*$-homomorphisms, and of the category of C*-algebras with $*$-homomorphisms as morphisms.
Note that if $A$ is a $G$-C*-algebra as above and
$\mathcal{M}(A)$ is the multiplier algebra of $A$, then we also obtain a canonical homomorphism $\overline{\alpha}\colon G \to \mathrm{Aut}(\mathcal{M}(A))$. We can thus talk about $G$-equivariant $*$-homomorphisms $\varphi \colon C\to \mathcal{M}(A)$
(here $C$ is another $G$-C*-algebra with the action $\gamma$) by saying $\varphi$ is
$G$-equivariant if
\begin{align*}
	\varphi (\gamma_s(c)) = \overline{\alpha}_s (\varphi(c))
\end{align*}
for all $s\in G$ and $c\in C$.

Let $-\rtimes_{u} G$  (resp.\ $-\rtimes_{r} G$) denote  the well-known
universal crossed product functor (resp.\ reduced crossed product functor) from the category of $G$-C*-algebras to the category of C*-algebras.
There is a natural transformation $\Lambda\colon -\rtimes_{u}G \Rightarrow -\rtimes_{r}G$
consisting of surjective $*$-homomorphisms such that
$\Lambda_A\vert_{A\rtimes_{\mathrm{alg}}G} = \mathrm{id}_{A\rtimes_{\mathrm{alg}}G}$.

The notion of a crossed product functor formalizes the idea of having  a `well-behaved' method of constructing `intermediate' crossed products.

\begin{defn}
	A \emph{crossed product functor} $-\rtimes_\mu G$ is a functor
from the category of $G$-C*-algebras to the category of C*-algebras, together with natural transformations
	\begin{align*}
		q\colon -\rtimes_u G \Rightarrow -\rtimes_\mu G \mbox{ and } s\colon	 -\rtimes_\mu G
		\rightarrow -\rtimes_r G
	\end{align*}
	consisting of surjective $*$-homomorphisms such that $s\circ 	q = \Lambda$.
\end{defn}
Note that the $*$-homomorphisms $q_A\colon A\rtimes_u G\to A\rtimes_\mu G$ and $s_A\colon
	A\rtimes_\mu G \to A\rtimes_r G$ for a $G$-C*-algebra $A$ are just the continuous
	extensions of the identity map $\mathrm{id}_{A\rtimes_{\mathrm{alg}} G}$. % To be more precise, the continuous extension of
	%$\mathrm{id}\colon (A\rtimes_{\mathrm{alg}}G,\|\cdot\|_{A\rtimes_{u}G})
%	\to (A\rtimes_{\mathrm{alg}}G,\|\cdot\|_{A\rtimes_\mu G} )$ and
%	$\mathrm{id}\colon (A\rtimes_{\mathrm{alg}}G,\|\cdot\|_{A\rtimes_{\mu}G})
%	\to (A\rtimes_{\mathrm{alg}}G,\|\cdot\|_{A\rtimes_r G} )$)
Whenever convenient, we will write also $\rtimes_\mu G$ instead of $-\rtimes_\mu G$.

Let $A$ be a $G$-C*-algebra. The universal crossed product $A\rtimes_u G$ comes with
a (universal) covariant representation $(\iota_A,\iota_G)\colon (A,G) \to \mathcal{M}(A\rtimes_u G)$.
Similarly we have the reduced/regular covariant representation $(\iota_{A,r},\iota_{G,r})\colon (A,G) \to \mathcal{M}(A\rtimes_r G)$.

Let $-\rtimes_\mu G$ be any crossed product functor. We write
\begin{align*}
	\iota_{A,\mu} &:= q_A \circ \iota_A \colon A \to \mathcal{M}(A\rtimes_\mu G);\\
	\iota_{G,\mu} &:= q_A \circ \iota_G \colon G\to \mathcal{M}(A\rtimes_\mu G).
\end{align*}
Furthermore, we write $C^*_\mu(G)$ for $\mathbb{C}\rtimes_\mu G$.
As is well known, the maps $\iota_{G}$, and respectively $\iota_{G,r}$, extend canonically to $*$-homomorphisms
\begin{align*}
	\iota_{G}\colon C^*_u(G) \to \mathcal{M}(A\rtimes_u G), \quad
	\iota_{G,r} \colon C^*_r(G) \to \mathcal{M}(A\rtimes_r G)
\end{align*}
given by
\begin{align*}
	(\iota_{G,\mu}(f) g)(t) = \sum_{s \in G} f(s)\alpha_s(g(s^{-1}t))
\end{align*}
for $f\in C_c(G)$, $g\in A\rtimes_{\mathrm{alg}}G$ and $t\in G$.

We are ready to formulate a first (relatively) weak property
for a crossed product functor.
\begin{defn}
	Let $-\rtimes_\mu G$ be a crossed product functor. We say that $-\rtimes_\mu G$
	\begin{itemize}
		\item[(i)] is \emph{functorial for generalised morphisms}
		if  for all $G$-C*-algebras $A$ and $B$ and every $G$-equivariant $*$-homomorphism
		$\varphi\colon A\to \mathcal{M}(B)$, there exists a $*$-homomorphism
		\begin{align*}
			\varphi\rtimes_\mu G \colon A\rtimes_\mu G \to \mathcal{M}(B\rtimes_\mu G)
		\end{align*}
		given by
		\begin{align*}
			((\varphi\rtimes_\mu G )(f) g )(t) =
			\sum_{s \in G} \varphi(f(s))\alpha_s(g(s^{-1}t))
		\end{align*}
		for all $f\in A\rtimes_{\mathrm{alg}}G $, $g\in B\rtimes_{\mathrm{alg}}G$ and
		$t\in G$;
		\item[(ii)] has the \emph{ideal property} if for all
		$G$-C*-algebras $A$ and all $G$-invariant closed ideals $I$ of $A$ the inclusion map
		$\iota\colon I \to A$ yields an injective $*$-homomorphism
		$\iota\rtimes_\mu G\colon I\rtimes_\mu G \to A\rtimes_\mu G$.
	\end{itemize}
\end{defn}

In fact, the properties above are equivalent.

\begin{prop}\cite[Lemma 3.3]{BEW}
	Let $-\rtimes_\mu G$ be a crossed product functor. Then $\rtimes_\mu G$ is functorial
	for generalised morphisms if and only if $\rtimes_\mu G$ has the ideal property.

\end{prop}
Both the universal $-\rtimes_u G$ and the reduced $-\rtimes_r G$ crossed product
	functors have the ideal property. We obtain the following simple observation.
\begin{prop}
	Let $-\rtimes_\mu G$ be a crossed product functor that is functorial for generalised morphisms,
and let	$A$ be a $G$-C*-algebra. Then the group homomorphism
	\begin{align*}
		\iota_{G,\mu}\colon G \to \mathcal{M}(A\rtimes_\mu G)
	\end{align*}
	extends canonically to a $*$-homomorphism
	\begin{align*}
		\iota_{G,\mu}\colon C_\mu^*(G) \to \mathcal{M}(A\rtimes_\mu G)
	\end{align*}
	given by
	\begin{align*}
		\iota_{G,\mu}(f) g(t) = \sum_{s \in G} f(s)\alpha_s(g(s^{-1}t))
	\end{align*}
	for $f\in C_c(G)$, $g\in A\rtimes_{\mathrm{alg}}G$ and $t\in G$.
\end{prop}
\begin{proof}
	Just use the obvious embedding $\mathbb{C}\to \mathcal{M}(A)$.
\end{proof}

The next property, dating back to \cite{KLQNYJM}, is very useful, as it allows us to use `good' functions on $G$ as the source of Schur multiplier maps on the relevant crossed products.

\begin{defn}
	Let $-\rtimes_\mu G$ be a crossed product functor. It is said to
	admit \emph{coactions} if for all $G$-C*-algebras $A$ the covariant representation
	\begin{align*}
		(\iota_{A,\mu}\otimes 1, \iota_{G,\mu}\otimes \iota_{G,u})\colon (A,G) \to
		\mathcal{M}(A\rtimes_\mu G \otimes C^*_u(G))
	\end{align*}
	yields a $*$-homomorphism
	\begin{align*}
	\beta_{A\rtimes_\mu G}:=	(\iota_{A,\mu}\otimes 1) \rtimes_\mu (\iota_{G,\mu}\otimes \iota_{G,u})\colon
		A\rtimes_\mu G \to \mathcal{M}(A\rtimes_\mu G \otimes C^*_u(G)).
	\end{align*}
 Here `yields' means that the $*$-homomorphism
	 $$(\iota_{A,\mu}\otimes 1) \rtimes_u (\iota_{G,\mu}\otimes \iota_{G,u})\colon
	A\rtimes_u G \to \mathcal{M}(A\rtimes_\mu G \otimes C^*_u(G))$$ associated to
	the covariant representation above factors through the quotient map
	$A\rtimes_u G \to A\rtimes_\mu G$.
\end{defn}

	Again, the universal and reduced crossed product functors do admit
    coactions.

\begin{rem}\label{rem:comultiplication}
	Suppose that $-\rtimes_\mu G$ is a crossed product functor that admits coactions. As a special
	case, we obtain a comultiplication like map
	\begin{align*}
		\delta_G = \iota_{G,\mu}\otimes \iota_{G}\colon C_\mu^*(G) \to \mathcal{M}(C^*_\mu (G)
		\otimes C^*_u(G)).
	\end{align*}
	Now, suppose $\phi\colon G \to \C$ is a normalised positive definite function on $G$ and
	let us  denote the corresponding state on $C^*_u(G)$ also by $\phi$. %To avoid some
%	technical issues let us assume that $G$ is discrete. Then $C^*_u(G)$ and $C^*_\mu (G)$
%	are unital C*-algebras and
Standard results on strict slice maps on multiplier algebras (\cite{Lance}) yield by a strict extension a completely positive map	$\mathrm{id}_{A\rtimes_\mu G }\otimes \phi \colon \mathcal{M}	(A\rtimes_\mu G  \otimes C^*_u(G)) \to
	\mathcal{M}(A\rtimes_\mu G )$. Define
	\begin{align*}
		m_\phi = (\mathrm{id}_{A\rtimes_\mu G }\otimes \phi) \circ \beta_{A\rtimes_\mu G } \colon A\rtimes_\mu G
		\to \mathcal{M}(A\rtimes_\mu G )
	\end{align*}
	and note that we have
	\begin{equation}\label{usualSchur}
		m_\phi( f) = \phi f\qquad(f \in 	A\rtimes_{\mathrm{alg}}G),
	\end{equation}
so that in fact the map $m_\phi$ takes values in $A\rtimes_\mu G $ (and naturally is determined by the formula \eqref{usualSchur} above). We shall see later that for even better behaved crossed product functors and certain choices of $\phi$ we can lift the map above to a map $\widetilde{m_\phi}\colon A\rtimes_\mu G
		\to A\rtimes_u G $.
\end{rem}

The next property, similarly to the ideal property, admits several equivalent descriptions.

\begin{defn}\label{defn:correspondence}
	Let $-\rtimes_\mu G$  be a crossed product functor. We say that $-\rtimes_\mu G$
	\begin{itemize}
		\item[(i)] has the \emph{cp-map property}, if for all $G$-C*-algebras $A$ and $B$ and every $G$-equivariant completely positive map
		$\varphi\colon A\to B$  there is a completely positive map $\varphi\rtimes_\mu G \colon
		A\rtimes_\mu G \to B\rtimes_\mu G$ given by
		\begin{align*}
			\varphi\rtimes_\mu G (f) = \varphi\circ f
		\end{align*}
		for all $f\in A\rtimes_{\mathrm{alg}}G$;
		\item[(ii)]  has the \emph{projection property}, if for all $G$-C*-algebras
		$A$ and all $G$-invariant projections $p \in \mathcal{M}(A)$ the inclusion
		$j\colon pAp \to A$ yields an inclusion $j\rtimes_\mu G\colon pAp\rtimes_\mu G \to A\rtimes_\mu G$;
		\item[(iii)]  has the \emph{hereditary subalgebra property}  if for all
		$G$-C*-algebras $A$ and all $G$-invariant hereditary subalgebras $B$ of $A$, the inclusion map
		$\iota\colon B \to A$ yields an injective $*$-homomorphism
		$\iota\rtimes_\mu G\colon B\rtimes_\mu G \to A\rtimes_\mu G$.
	\end{itemize}
\end{defn}

\begin{thm}\cite[Theorem 4.9]{BEW} \label{thm:char_correspondence_func}
	All the properties in Definition \ref{defn:correspondence} are equivalent. A crossed product
	functor that satisfies these  properties  is
	called a \emph{correspondence functor}.
\end{thm}

Both the reduced $-\rtimes_r G $ and the universal $-\rtimes_u G$ crossed product functors are correspondence functors. The first part of the next fact is also contained in \cite[Theorem 4.9]{BEW}, and the second is \cite[Theorem 5.6]{BEW}.

\begin{thm}
	A correspondence functor has the ideal property and admits coactions.
\end{thm}

The next two results are proved in \cite{AEES}: the first is \cite[Proposition 2.8]{AEES}, the second is \cite[Lemma 2.10]{AEES}.

\begin{prop}\label{prop:triv}
	Let $-\rtimes_\mu G$ be a correspondence crossed product functor and let $A$ be a C*-algebra equipped with a trivial $G$-action. Then the images of $*$-homomorphisms
	$\iota_{A,\mu}\colon A\to \mathcal{M}(A\rtimes_\mu G)$ and
	$\iota_{G,\mu}\colon C_\mu^*(G) \to \mathcal{M}(A\rtimes_\mu G)$ commute,
	$\iota_{A,\mu}(a)\iota_{G,\mu}(b) \in A\rtimes_\mu G$ for all
	$a\in A$ and $b\in C^*_\mu(G)$ and the $*$-homomorphism
	\begin{align*}
		\iota_{A,\mu}\odot \iota_{G,\mu}\colon A\odot C_\mu^*(G) \to A\rtimes_\mu G
	\end{align*}
	given by
	\begin{align*}
		\iota_{A,\mu}\odot \iota_{G,\mu}(a\otimes b) =\iota_{A,\mu}(a)\iota_{G,\mu}(b)
	\end{align*}
	for $a\in A$ and $b\in C_\mu^*(G)$ is injective and has dense image.
	
	In particular, $A\rtimes_\mu G$ is a C*-tensor product of $A$ and $C^*_\mu(G)$.
	We will write $A\otimes_\mu C^*_\mu(G)$ for $A\rtimes_\mu G$.
\end{prop}

\begin{lem}\label{lem:strong_fell_absorption}
		Let $-\rtimes_\mu G$ be a correspondence crossed product functor and let $A$ be a $G$-C*-algebra.
		The $*$-homomorphisms
		 $\iota_A \otimes_\mu 1\colon A \to \mathcal{M}((A\rtimes_{u} G) \otimes_\mu C_\mu^* (G))$
		induced by
		\begin{align*}
			(\iota_A \otimes_\mu 1)(a) (b\otimes c) = \iota_A(a)b \otimes c
		\end{align*}
		for $a\in A$, $b\in A\rtimes_{u} G$ and $c\in C^*_\mu (G)$,
		and
		$\iota_G \otimes_\mu \iota_{G,\mu}\colon G \to \mathcal{M}((A\rtimes_{u} G) \otimes_\mu C_\mu^* (G))$
		induced by
		\begin{align*}
			(\iota_G \otimes_\mu \iota_{G,\mu})(s) (b\otimes c) = \iota_G(s)b \otimes \iota_{G,\mu}(s) c
		\end{align*}
		for all $s\in G$, $b\in A\rtimes_{u} G$ and $c\in C_\mu^* (G)$
		define a nondegenerate covariant representation of $A$. Its integrated form
		\[(\iota_A \otimes_\mu 1) \rtimes_{u} (\iota_G \otimes_\mu \iota_{G,\mu}) \colon A\rtimes_{u} G \to \mathcal{M}((A\rtimes_{u} G) \otimes_\mu C_\mu^* (G) )\]
		factors through
		$q_A \colon A\rtimes_{u} G \to A\rtimes_\mu G$, i.e., there is a unique injective $*$-homomorphism
		\[\gamma_{A\rtimes_\mu G}:=(\iota_A \otimes_\mu 1) \rtimes_\mu (\iota_G \otimes_\mu \iota_{G,\mu})\colon A\rtimes_\mu G
		\to \mathcal{M}((A\rtimes_{u} G) \otimes_\mu C_\mu^* (G))\]
		such that $(\iota_A \otimes_\mu 1) \rtimes_{u} (\iota_G \otimes_\mu \iota_{G,\mu})
		= ((\iota_A \otimes_\mu 1) \rtimes_\mu (\iota_G \otimes_\mu \iota_{G,\mu})) \circ q_A$.
	\end{lem}

The above lemma allows us to construct the `$\mu$-to-full' Schur multiplier maps mentioned in Remark \ref{rem:comultiplication}.

\begin{prop} \label{Schurrevisited}
   Let $-\rtimes_\mu G$ be a correspondence crossed product functor and let  $\phi \in C^*_\mu(G)^*\subseteq B(G)$. Then the prescription
\begin{equation}\label{usualSchur2}
\widetilde{m_\phi}( f) = \phi f\qquad(f \in 	A\rtimes_{\mathrm{alg}}G),
\end{equation}
extends to a completely bounded map (denoted by the same symbol)
\[
\widetilde{m_\phi}  \colon A\rtimes_\mu G
\to A\rtimes_u G,
\]
with $\|\widetilde{m_\phi}\|_{cb} \leq \|\phi\|_{C^*_\mu(G)^*}$.
\end{prop}

\begin{proof}
Strict extension yields a completely bounded strict map	$\mathrm{id}_{A\rtimes_u G }\otimes \phi \colon \mathcal{M}	(A\rtimes_u G  \otimes C^*_\mu(G)) \to
\mathcal{M}(A\rtimes_u G )$. Use Lemma \ref{lem:strong_fell_absorption} to define a bounded map
\begin{align*}
\widetilde{m_\phi} = (\mathrm{id}_{A\rtimes_u G }\otimes \phi) \circ q \circ  \gamma_{A\rtimes_\mu G } \colon A\rtimes_\mu G
\to \mathcal{M}(A\rtimes_u G ),
\end{align*}
where $q\colon(A\rtimes_u G)  \otimes_\mu C^*_\mu(G)\to (A\rtimes_u G)  \otimes C^*_\mu(G)$ is the obvious quotient map; its existence follows from Takesaki's theorem saying that the spatial norm is the minimal $C^*$-norm on the algebraic tensor product. It is easy to check that   $ \widetilde{m_\phi}$ satisfies the condition \eqref{usualSchur2} and hence actually takes values in $A\rtimes_u G$.
\end{proof}

Finally we mention an easy observation which tells us that when the action of $G$ on $A$ admits an invariant state, and $\rtimes_\mu$ is a $G$-correspondence functor, then the crossed product $A \rtimes_\mu G$ `remembers' $C^*_\mu(G)$.

%{\color{red} unitality is likely not needed, will think about it. \color{blue} Maybe we don't need it. The unital case, I guess, is good enogh for us.}

\begin{prop}\label{invariantstate}
	%Let $G$ be a locally compact group
	Let $-\rtimes_\mu G$ be a correspondence crossed product functor, and let $G$ act on a unital C*-algebra $A$ with a $G$-invariant state.	Then the canonical map $\iota_{G,\mu}\colon C^*_\mu(G) \to A \rtimes_\mu G$ is isometric.
\end{prop}
\begin{proof}
A  $G$-invariant state is, in particular, a $G$-invariant completely positive map from $A$ to $\mathbb{C}$ (which admits a $G$-invariant right inverse given by $\lambda \mapsto \lambda 1_A$) so that we obtain a completely positive, contractive map $\phi_\mu\colon A \rtimes_\mu G \to 	C^*_\mu(G)$, such that $\iota_{G,\mu}$ is its right inverse. Hence $\iota_{G,\mu}$ is isometric.
\end{proof}

For examples of crossed product functors, notably the Brown-Guentner and Kaliszewski-Landstad-Quigg functors and their properties, we refer to \cite{BEW}. Note already here that the Brown-Guentner crossed product functor, whilst closely related to the functors we consider in this paper, is not a correspondence functor (which will turn out to be the case for our functors), as discussed in \cite[Remark 4.19]{BEW}.

\section{Banach *-algebras $\textup{PF}_p^*(G;A)$ and their properties} \label{sec:Banach}

As in \cite{SW2}, the  definition of the crossed product functor in our work will use at a certain point the enveloping C*-algebra construction. In this section, we introduce and discuss the relevant Banach $*$-algebras.

\subsection{Definition of the $\ell^p$-convolution crossed product Banach algebra $\textup{PF}_p^*(G;A)$}

Let $(A,\alpha)$ be a $G$-C*-algebra. A \emph{(Banach space) covariant representation} is a triple
$(X,\pi, U)$ consisting of a Banach space $X$, a contractive algebra-homomorphism
$\pi\colon A \to \mathcal{B}(X)$ and a group homomorphism
$U\colon G \to \mbox{Isom}(X)$ into the isometry group $\mbox{Isom}(X)$ of $X$, such that
\begin{align*}
	U_s \pi(a) = \pi(\alpha_s(a)) U_s
\end{align*}
holds for all $s\in G$ and $a\in A$. We say $(X,\pi, U)$ is a \emph{C*-covariant representation}
if $(X,\pi, U)$ is a covariant representation, $X$ an Hilbert space, $\pi$ is a $*$-homomorphism
and $U(G) \subseteq \mathcal{U}(X)$, where $\mathcal{U}(X)$ denotes the group
of unitary operators on $X$.

\begin{rem}
	i) Let $\pi\colon A \to \mathcal{B}(\mathcal{H})$ be a $*$-representation of the $G$-C*-algebra $A$
	on a Hilbert space $\mathcal{H}$.
	Let $\ell^p(G;\mathcal{H})$ be the  $\mathcal{H}$-valued $\ell^p$-space for a
	$p\in [1,\infty)$.
	For every $a\in A$ and $f\in \ell^p(G;\mathcal{H})$ we define
	\begin{align*}
		\widetilde{\pi}_p(a)f\colon G \to \mathcal{H},\quad s\mapsto \pi(\alpha_{s^{-1}}(a))(f(s)).
	\end{align*}
	The inequality $\| (\widetilde{\pi}_p(a)f)(s) \| \leq \| a\| \|f(s)\|$ for all $s\in G$ implies
	\begin{align}\label{cp:eq:regrep1}
		\|\widetilde{\pi}_p(a)f \|_p \leq \|a\| \|f\|_p,
	\end{align}
	so $\widetilde{\pi}_p(a)f\in \ell^p(G;\mathcal{H})$. The inequality \eqref{cp:eq:regrep1}
	yields further  that
	$\widetilde{\pi}_p(a)\colon \ell^p(G; \mathcal{H})\to \ell^p(G;\mathcal{H})$ is
	a well defined linear operator with norm bounded by $\|a\|$. We thus obtain a contractive
	algebra-homomorphism $\widetilde{\pi}_p\colon A \to \mathcal{B}(\ell^p(G;\mathcal{H}))$.
	Now let $p>1$ and let $q \in (1,\infty)$ be the `conjugate index', i.e.~$1/p+1/q=1$. For $f \in \ell^p(G;\mathcal{H})$, $g\in \ell^q(G;\mathcal{H})$, write $\la f,g\ra = \sum_{t\in G} \la f(t), g(t)\ra $ for the canonical \emph{sesquilinear} pairing. The operator $\widetilde{\pi}_p$ is still `exterior involutive', i.e.
	\begin{align}\label{cp:eq:exinv}
		\langle \widetilde{\pi}_p(a)f,g \rangle = \sum_{s\in G} \langle \pi(\alpha_{s^{-1}}(a))f(s),g(s) \rangle
		 = \sum_{s\in G} \langle f(s), \pi(\alpha_{s^{-1}}(a^*))g(s) \rangle
		= \langle f, \widetilde{\pi}_q(a^*)g \rangle.
	\end{align}
	Let
	\begin{align*}
		\lambda^p \colon G \to \mbox{Isom}(\ell^p(G;\mathcal{H}))
	\end{align*}
	be the group homomorphism given by $(\lambda^p(s)f)(t) = f(s^{-1}t)$ for $f\in \ell^p(G;\mathcal{H})$ and $s,t\in G$.
	Note that for $f\in \ell^p(G;\mathcal{H})$, $g\in \ell^q(G; \mathcal{H})$ and $s \in G$ we have
	\begin{align}\label{cp:eq:regrepinv}
		\langle \lambda^p(s)f,g\rangle = \sum_{t\in G} \langle f(s^{-1}t),g(t) \rangle
		= \sum_{t\in G} \langle f(t) ,g(st)\rangle  = \langle f, \lambda^q(s^{-1}) g \rangle.
	\end{align}
It is straightforward to check that $(\ell^p(G;\mathcal{H}),\widetilde{\pi}_p,
	\lambda^p)$ is a covariant representation. We will call this construction the
	\emph{$\pi$-valued $p$-regular covariant representation}. We will occasionally omit the index $p$ from $\lambda^p$ or $\widetilde{\pi}_p$.
	
	ii) Let $\varphi\colon A \to B$ be a $G$-equivariant $*$-homomorphism between
	$G$-C*-algebras, let $\pi\colon B \to \mathcal{B}(\mathcal{H})$ be a $*$-representation of $B$ on
	an Hilbert space $\mathcal{H}$ and $p\in [1,\infty)$. Furthermore,
	let the triple $(\ell^p(G;\mathcal{H}), \widetilde{\pi}_p, \lambda^p)$ be the
	$\pi$-valued $p$-regular covariant representation. Then
	$(\ell^p(G;\mathcal{H}),\widetilde{\pi}_p\circ \varphi, \lambda^p) =
	(\ell^p(G;\mathcal{H}),\widetilde{(\pi\circ \varphi)_p}, \lambda^p)$.
\end{rem}

Let $(A,\alpha)$ be a $G$-C*-algebra and let $(X,\pi, U)$ be a covariant representation of $(A,\alpha)$. For $f\in \ell^1(G;A)$, using the inequality
$\|\pi(f(s)) U_s x \| \leq \|f(s)\| \|x\|$ valid for all $s\in G,x\in X$, we see that the function
$G \ni s \mapsto \pi(f(s))U_s x \in X$ is integrable and
\begin{align*}
	\pi \rtimes U(f) \colon X \to X,\quad x \mapsto \sum_{s\in G} \pi(f(s))U_s x
\end{align*}
is a linear operator, with the norm bounded by $\|f\|_1$. This yields a contractive Banach algebra homomorphism
\begin{align*}
	\pi\rtimes U \colon \ell^1(G;A) \to \mathcal{B}(X),\quad  f \mapsto \pi\rtimes U(f).
\end{align*}

\begin{prop}\label{prop:pointwiseext}
	Let $A$ and $B$ be two $G$-C*-algebras and let $\varphi\colon A \to B$ be a
	$G$-equivariant $*$-homo\-mor\-phis\-m. Then the `pointwise extension' of $\varphi$, denoted  $\varphi\rtimes_1 G$, yields  a contractive $*$-homo\-mor\-phism   from $\ell^1(G;A)$ to $\ell^1(G; B)$.
	If $\pi \colon B\to \mathcal{B}(\mathcal{H})$ is a $*$-representation, $p\in [1,\infty)$ and
	$(\ell^p(G;\mathcal{H}),\tilde{\pi}_p,\lambda^p)$ is the $\pi$-valued
	$p$-regular covariant representation, then
	$$\tilde{\pi}\rtimes \lambda^p((\varphi\rtimes_1 G) f) = \widetilde{\pi\circ \varphi}\rtimes \lambda^p
	(f)$$ for all $f\in \ell^1(G;\mathcal{H})$.
\end{prop}
\begin{proof}
	Let $f\in \ell^1(G;A)$ and $g\in \ell^p(G; \mathcal{H})$. Then, using
	that $\tilde{\pi}(\varphi(a)) = \widetilde{\pi\circ \varphi}(a)$ for all
	$a\in A$, we get
	\begin{align*}
 \tilde{\pi}\rtimes \lambda^p((\varphi\rtimes_1 G) f) g
		= \sum_{s\in G} \tilde{\pi}(\varphi(f(s))) \lambda^p(s) g
		= \sum_{s\in G} \widetilde{\pi\circ \varphi}(f(s)) \lambda^p(s) g
		= \widetilde{\pi\circ \varphi} \rtimes \lambda^p(f)  g.
	\end{align*}
\end{proof}
\begin{rem}
	Let $p\in (1,\infty)$, let $A$ be a  $G$-C*-algebra and let $\pi\colon A \to \mathcal{B}
	(\mathcal{H})$ be a $*$-representation of $A$ on
	a Hilbert space $\mathcal{H}$.  Let $(\ell^p(G,\mathcal{H}), \widetilde{\pi}_p, \lambda^p)$
	be the $\pi$-valued $p$-regular covariant representation.
	Let $q\in (1,\infty)$ with $1/p+1/q = 1$ and let $f\in \ell^p(G;\mathcal{H})$
	and $g\in \ell^q(G;\mathcal{H})$. Then for any $h\in \ell^1(G;A)$ using
	equations \eqref{cp:eq:exinv} and \eqref{cp:eq:regrepinv} we obtain
	\begin{align*}
		\langle \widetilde{\pi}\rtimes \lambda^p(h) f, g \rangle
		&= \sum_{s\in G} \langle \widetilde{\pi}_p(h(s)) \lambda^p(s) f,g \rangle
		= \sum_{s\in G}  \langle \lambda^p(s)f, \widetilde{\pi}_q(h(s)^*)g\rangle \\
		&= \sum_{s \in G}	 \langle \lambda^p(s^{-1})f,\widetilde{\pi}_q(h(s^{-1})^*)g\rangle
		= \sum_{s\in G}	\langle f, \widetilde{\pi}_q(\alpha_s(h(s^{-1})^*))\lambda^q(s)g \rangle \\
		&= \langle f, \widetilde{\pi}\rtimes \lambda^q(h^*) g \rangle.
	\end{align*}
	Thus, for all $h\in \ell^1(G;A)$ and $f\in \ell^p(G;\mathcal{H})$,
	$g\in \ell^q(G;\mathcal{H})$ we obtain
	\begin{align}
		\langle \widetilde{\pi}\rtimes \lambda^p(h) f, g \rangle
		= \langle f, \widetilde{\pi}\rtimes \lambda^q(h^*) g \rangle.
	\end{align}
Using the standard fact that $\|f\|_p=\sup\{|\la f,g\ra | \mid g\in \ell^q(G;\mathcal{H}): \|g\|_q\le 1\}$ we arrive at the equalities
	\begin{align*}
&\quad\;
\| \widetilde{\pi}\rtimes \lambda^p(h^*) \|_{\mathcal{B}(\ell^p(G;\mathcal{H}))}\\
&=
		\sup\lbrace |\langle \tilde{\pi}\rtimes \lambda^p(h^*)f,g\rangle| \mid
		f\in \ell^p(G;\mathcal{H}),\, g\in \ell^q(G; \mathcal{H}): \, \|f\|_p, \, \|g\|_q\leq 1
		\rbrace\\
		&= \sup\lbrace |\langle f,\tilde{\pi}\rtimes \lambda^q(h) g\rangle| \mid
		f\in \ell^p(G;\mathcal{H}),\, g\in \ell^q(G; \mathcal{H}): \, \|f\|_p, \, \|g\|_q\leq 1
		\rbrace
		\\
		& =\| \widetilde{\pi}\rtimes \lambda^q(h)\|_{\mathcal{B}(\ell^q(G;\mathcal{H}))}.
	\end{align*}
\end{rem}

This allows us to consider the following Banach $*$-algebras which could be regarded as vector valued analogues of $\textup{PF}^*_p(G)$ considered in \cite{SW1} and \cite{SW2}.

\begin{defn}
	Let $A$ be a $G$-C*-algebra, let $p\in (1,\infty)$ with conjugate index $q$ and let $\pi\colon A \to \mathcal{B}(\mathcal{H})$
	be a $*$-representation of $A$ on a Hilbert space $\mathcal{H}$. Set
	\begin{align*}
		\| \cdot \|_{\pi,p}\colon  \ell^1(G; A )\to [0,\infty),\quad f \mapsto	 \mbox{max}\lbrace
		\| \widetilde{\pi}\rtimes \lambda^p(f)\|_{\mathcal{B}(\ell^p(G;\mathcal{H}))},\,
		\| \widetilde{\pi}\rtimes \lambda^q(f)\|_{\mathcal{B}(\ell^q(G;\mathcal{H}))}\rbrace
	\end{align*}
and
	\begin{align*}
		\|\cdot \|_{\textup{PF}^*_p(G;A)}\colon \ell^1(G;A)\to [0,\infty), \quad f \mapsto \sup
		\lbrace \| f \|_{\pi,p} \mid (\mathcal{H},\pi) \mbox{ is a} \ast\mbox{-representation of}\ A  	\rbrace.
	\end{align*}
\end{defn}

The next proposition is obvious.

\begin{prop}
	Let $A$ be a $G$-C*-algebra and $p\in (1,\infty)$. Then if $\pi$ is a $*$-representation of $A$ and $\pi'$ denotes the essential part of $\pi$,
	then for every $f\in \ell^1(G;A)$ we have
	\begin{align*}
		\|\pi \rtimes \lambda^p(f)\| = \| \pi'\rtimes \lambda^p(f)\|.
	\end{align*}
Consequently, the equality
	\begin{align*}
		\| \cdot \|_{\PF{p}{G}{A} } = \sup\lbrace \| \cdot \|_{\pi, p} \mid \pi \mbox{ is
		a non-degenerate} \ast\mbox{-representation of}\ A \rbrace
	\end{align*}
	holds. Furthermore, if $\pi$ is faithful, then $\|\cdot\|_{\pi,p}$ is a norm on $\ell^1(G;A)$.
\end{prop}

\begin{defn}
	Let $1<p< \infty$, let $A$ be a $G$-C*-algebra and let $\pi\colon A\rightarrow \mc{B}(\mc{H})$ be a faithful $*$-representation. We let $\PF{p,\pi}{G}{A}$ and $\PF{p}{G}{A}$ be the completion of $\ell^1(G;A)$ with respect to norms, respectively, $\|\cdot\|_{\pi,p}$ and $\| \cdot\|_{\PF{p}{G}{A}}$.
\end{defn}

\begin{rem} \label{rem:contractive}
It is clear that $\PF{2}{G}{A}=A \rtimes_r G$ and $\PF{p}{G}{A}=\PF{q}{G}{A}$ when $p$ and $q$ are conjugates. Moreover the discussion before Proposition \ref{prop:pointwiseext} implies that we always have a natural contractive and injective $*$-algebra map $\iota_{A,p}\colon\ell^1(G;A) \to \PF{p}{G}{A}$, with dense image.
\end{rem}

\subsection{(In)dependence of the algebras $\PF{p,\pi}{G}{A}$ of the choice of the representation $\pi$}

We shall show that the Banach algebra $\PF{p}{G}{\mathbb{C}}$ is canonically isometrically isomorphic to the Banach algebra $\textup{PF}_p^*(G)$ introduced in \cite[Definition 4.3]{SW1} and $\PF{p,\pi}{G}{A}$ is \emph{isomorphic} to $\PF{p}{G}{A}$ for any faithful representation $\pi$ of $A$. As a consequence, we shall later obtain an isomorphism statement for the corresponding C$^*$-algebras.

First we need to establish a lemma, in whose proof we will use the family $\{r_n\}_{n\in\N}\subseteq \LL^{\infty}([0,1])$ of Rademacher functions (see \cite[Page 10]{DJT}). The Rademacher functions form an orthonormal set in $\LL^2([0,1])$. We will also use the Kahane's inequality \cite[Theorem 11.1]{DJT}, which says that for any $1\le p,q<\infty$ there is a constant $K_{p,q}>0$ such that
\begin{equation} \label{KKInequality}
\bigl(\int_{0}^{1}\bigl\|\sum_{k=1}^{n} x_k r_k(t)\bigr\|^q \md t\bigr)^{1/q}\le
K_{p,q}
\bigl(\int_{0}^{1}\bigl\|\sum_{k=1}^{n} x_k r_k(t)\bigr\|^p \md t\bigr)^{1/p}
\end{equation}
for any Banach space $X$ and vectors $x_{1},\dotsc,x_n\in X$. Proof of Lemma \ref{lemma1} bears some resemblance with a proof of Marcinkiewicz-Zygmund theorem. However, there are non-trivial differences between both arguments: we need to use vector valued inequality  \eqref{KKInequality} and we are in a more general situation, as our operators act in a non-trivial way on the underlying space. In particular, one cannot easily obtain Lemma \ref{lemma1} as a corollary of the Marcinkiewicz-Zygmund theorem.
%Of course we can take $K_{p,q}=1$ when $q\le p$.

\begin{lem}\label{lemma1}
Let $1\le p<\infty$, $A$ be a $G$-C$^*$-algebra, $\pi$ a non-degenerate $*$-representation of $A$ on a Hilbert space $\mc{H}_\pi$ and $\sigma=\pi(\cdot)\otimes I_{\mc{K}}\colon A\rightarrow \mc{B}(\mc{H}\otimes \mc{K})$ its amplification for a non-zero Hilbert space $\mc{K}$. Then there exists a constant $C_p\ge 1$ depending only on $p$, such that
\begin{equation}\label{eq3}
\|(\widetilde{\pi}_p\rtimes \lambda^p)(f)\|_{\mc{B}(\ell^p(G;\mc{H}))}\le
\|(\widetilde{\sigma}_p\rtimes \lambda^p)(f)\|_{\mc{B}(\ell^p(G;\mc{H}\otimes \mc{K}))}\le
C_p \|(\widetilde{\pi}_p\rtimes \lambda^p)(f)\|_{\mc{B}(\ell^p(G;\mc{H}))}
\end{equation}
for $f\in\ell^1(G;A).$
\end{lem}

\begin{proof}
Since representations $\widetilde{\pi}_p\rtimes \lambda^p, \widetilde{\sigma}_p\rtimes \lambda^p$ are contractive, it is enough to prove \eqref{eq3} for a fixed finitely supported function $f\colon G\rightarrow A$.

The first inequality in \eqref{eq3} is straightforward: take $\xi\in \ell^p(G;\mc{H})$, choose a unit vector $\zeta\in\mc{K}$ and define $\xi_\zeta\colon G\ni g\mapsto \xi(g)\otimes \zeta\in \mc{H}\otimes\mc{K}$. Then $\|\xi_\zeta\|_{\ell^p(G;\mc{H}\otimes \mc{K})}=\|\xi\|_{\ell^p(G;\mc{H})}$ and
\begin{align*}
\|(&\wt{\pi}_p\rtimes\lambda^p)(f)\xi\|_{\ell^p(G;\mc{H})}=
\bigl(\sum_{h\in G}\bigl\|\sum_{g\in G}  \pi\bigl(\alpha_{h^{-1}}(f(g))\bigr)\xi(g^{-1} h)\bigr\|^p_{\mc{H}}\bigr)^{1/p}\\
&=
\bigl(\sum_{h\in G}\bigl\|\sum_{g\in G}  \pi\bigl(\alpha_{h^{-1}}(f(g))\bigr)\xi(g^{-1} h)\otimes \zeta\bigr\|^p_{\mc{H}\otimes \mc{K}}\bigr)^{1/p}=
\bigl(\sum_{h\in G}\bigl\|
((\wt{\sigma}_p\rtimes\lambda^p)(f) \xi_\zeta)(h)
\bigr\|^p_{\mc{H}\otimes \mc{K}}\bigr)^{1/p}\\
&=
\bigl\|
(\wt{\sigma}_p\rtimes\lambda^p)(f) \xi_\zeta\|_{\ell^p(G;\mc{H}\otimes \mc{K})} \le
\|(\wt{\sigma}_p\rtimes\lambda^p)(f)\| \|\xi\|_{\ell^p(G;\mc{H})},
\end{align*}
which proves the first inequality in \eqref{eq3}. Fix $\xi\in \ell^p(G;\mc{H}\otimes \mc{K})$ with norm $\le 1$. To establish the second inequality, we need to prove that
\begin{equation}\label{eq4}
\|(\widetilde{\sigma}_p\rtimes \lambda^p)(f)\xi\|_{\ell^p(G;\mc{H}\otimes \mc{K})}\le
C_p \|(\widetilde{\pi}_p\rtimes \lambda^p)(f) \|_{\ell^p(G;\mc{H})}
\end{equation}
with a constant $C_p$ depending only on $p$. Since $\xi$ is $p$-summable, there is a separable Hilbert space $\mc{K}_0\subseteq \mc{K}$ such that $\xi(h)\in \mc{H}\otimes \mc{K}_0$ for $h\in G$. Then we easily see that \eqref{eq4} is equivalent to
\begin{equation}\label{eq5}
\|(\widetilde{\sigma}_p\rtimes \lambda^p)(f)\xi\|_{\ell^p(G;\mc{H}\otimes \mc{K}_0)}\le
C_p \|(\widetilde{\pi}_p\rtimes \lambda^p)(f) \|_{\ell^p(G;\mc{H})},
\end{equation}
where by a slight abuse of the notation we write also $\sigma$ for the amplification of $\pi$ on $\mc{H}\otimes \mc{K}_0$. By passing via isometry, we can in fact assume that $$ \mc{K}_0=\mc{R}=\ov{\lin}\{r_n\mid n\in\N\}\subseteq \LL^2(\left[0,1\right]).$$ For $h\in G$, we write
\[
\xi(h)=
\sum_{n=1}^{\infty}  \eta_n(h)\otimes r_n,\quad
\textnormal{where}\;\;\;
\eta_n\colon G\ni h\mapsto(\mathrm{id}\otimes   r_n^*) (\xi(h))\in \mc{H}.
\]
The above series converges in the norm of $\mc{H}\otimes \mc{R}$ for every $h\in G$. Then also
\[
\bigl\|\xi-\sum_{n=1}^{N} \eta_n(\cdot)\otimes r_n\bigr\|_{\ell^p(G;\mc{H}\otimes\mc{R})}^p
=
\sum_{h\in G} \bigl\|
\sum_{n=N+1}^{\infty} \eta_n(h)\otimes r_n\bigr\|_{\mc{H}\otimes \mc{R}}^p \xrightarrow[N\to\infty]{}0
\]
by the bounded convergence theorem, where the uniform bound is given by
\[
\bigl\|
\sum_{n=N+1}^{\infty} \eta_n(h)\otimes r_n\bigr\|_{\mc{H}\otimes \mc{R}}^p=
\bigl(\sum_{n=N+1}^{\infty} \bigl\|\eta_n(h)\bigr\|_{\mc{H}}^{2 }\bigr)^{p/2}\le
\bigl(\sum_{n=1}^{\infty} \bigl\|\eta_n(h)\bigr\|_{\mc{H}}^{2}\bigr)^{p/2}=
\|\xi(h)\|^{p}.
\]
This proves that $\sum_{n=1}^{N} \eta_n(\cdot)\otimes r_n\xrightarrow[N\to\infty]{}\xi$ in $\ell^p(G;\mc{H}\otimes\mc{R})$. With these preliminaries we calculate the left hand side of \eqref{eq5}
\begin{align*}
\|(\wt{\sigma}_p\rtimes\lambda^p)(f) \xi\|_{\ell^p(G;\mc{H}\otimes \mc{R})}&=
\lim_{N\to\infty}
\bigl\|(\wt{\sigma}_p\rtimes\lambda^p)(f)
\bigl(
\sum_{n=1}^{N} \eta_n(\cdot)\otimes r_n
\bigr)
\bigr\|_{\ell^p(G;\mc{H}\otimes \mc{R})}
\\
&=
\lim_{N\to\infty}
\bigl(
\sum_{h\in G}
\bigl\|
\sum_{g\in G}
\bigl(
\wt{\sigma}_p(f(g))\lambda^p(g)
\bigl(\sum_{n=1}^{N}\eta_n(\cdot)\otimes r_n\bigr)
\bigr)(h)\bigr\|^p_{\mc{H}\otimes\mc{R}}\bigr)^{1/p}\\
&=
\lim_{N\to\infty}
\bigl(
\sum_{h\in G}
\bigl\|
\sum_{g\in G}
\sum_{n=1}^{N}
 \pi (\alpha_{h^{-1}}(f(g)))
\eta_n(g^{-1}h)\otimes r_n
\bigr\|^p_{\mc{H}\otimes\mc{R}}\bigr)^{1/p}.
\end{align*}
Note that the sum over $g\in G$ is finite. As $\mc{R}\subseteq \LL^2([0,1])$, we can calculate the last norm in $\mc{H}\otimes \LL^2([0,1])\simeq \LL^2([0,1]; \mc{H})$, using \eqref{KKInequality}:
\begin{align*}
\|(\wt{\sigma}_p\rtimes\lambda^p)(f) \xi&\|_{\ell^p(G;\mc{H}\otimes \mc{R})}\\
&= \lim_{N\to\infty}
\bigl(
\sum_{h\in G}
\bigl(
\int_{0}^{1}
\bigl\|
\sum_{g\in G}
\sum_{n=1}^{N}
 \pi (\alpha_{h^{-1}}(f(g)))
\eta_n(g^{-1}h)r_n(t)
\bigr\|^2_{\mc{H}}
\md t
\bigr)^{p/2}\bigr)^{1/p}\\
&\le
\limsup_{N\to\infty}
\bigl(
\sum_{h\in G}
\bigl(
K_{p,2}
\bigl[
\int_{0}^{1}
\bigl\|
\sum_{g\in G}
\sum_{n=1}^{N}
 \pi(\alpha_{h^{-1}}(f(g)))
\eta_n(g^{-1}h)r_n(t)
\bigr\|^p_{\mc{H}}
\md t
\bigr]^{1/p}\bigr)^{p}
\bigr)^{1/p}\\
&=
K_{p,2}
\limsup_{N\to\infty}
\bigl(
\int_{0}^{1}
\sum_{h\in G}
\bigl\|
\sum_{n=1}^{N} r_n(t)
((\wt{\pi}_p\rtimes\lambda^p)(f)\eta_n)(h) \bigr\|^p_{\mc{H}}
\md t
\bigr)^{1/p}.
\end{align*}
We have switched the order of integral and sum, since the integrand is non-negative and we can restrict to a countable subset of $G$. Then
\begin{align*}
\|(\wt{\sigma}_p\rtimes\lambda^p)(f) \xi&\|_{\ell^p(G;\mc{H}\otimes \mc{R})}\\
&\le
K_{p,2}
\limsup_{N\to\infty}
\bigl(
\int_{0}^{1}
\bigl\|
\sum_{n=1}^{N} r_n(t)
(\wt{\pi}_p\rtimes\lambda^p)(f)\eta_n \bigr\|^p_{\ell^p(G;\mc{H})}
\md t
\bigr)^{1/p}\\
&\le
K_{p,2}
\|(\wt{\pi}_p\rtimes\lambda^p)(f)\|
\limsup_{N\to\infty}
\bigl(
\int_{0}^{1}
\bigl\|
\sum_{n=1}^{N} r_n(t)
 \eta_n \bigr\|^p_{\ell^p(G;\mc{H})}
\md t
\bigr)^{1/p}\\
&=
K_{p,2}
\|(\wt{\pi}_p\rtimes\lambda^p)(f)\|
\limsup_{N\to\infty}
\bigl(
\sum_{h\in G}
\int_{0}^{1}
\bigl\|
\sum_{n=1}^{N} r_n(t)
 \eta_n(h) \bigr\|^p_{ \mc{H}}
\md t
\bigr)^{1/p}\\
&\le
K_{p,2}
\|(\wt{\pi}_p\rtimes\lambda^p)(f)\|
\limsup_{N\to\infty}
\bigl(
\sum_{h\in G}
\bigl(
K_{2,p}\bigl[\int_{0}^{1}
\bigl\|
\sum_{n=1}^{N} r_n(t)
 \eta_n(h) \bigr\|^2_{ \mc{H}}
\md t\bigr]^{1/2}\bigr)^{p}
\bigr)^{1/p}\\
&=
K_{p,2} K_{2,p}
\|(\wt{\pi}_p\rtimes\lambda^p)(f)\|
\limsup_{N\to\infty}
\bigl\|
\sum_{n=1}^{N}
 \eta_n(\cdot)\otimes r_n \bigr\|_{\ell^p(G; \mc{H}\otimes \mc{R})}\\
 &=
K_{p,2} K_{2,p}
\|(\wt{\pi}_p\rtimes\lambda^p)(f)\|
\|
\xi \|_{\ell^p(G; \mc{H}\otimes \mc{R})}.
\end{align*}
This ends the proof with $C_p=K_{p,2} K_{2,p}$.
\end{proof}

Next we show that in the case of $A=\C$ and trivial action, the constant in the above lemma can be improved to $1$. The general method behind our proof is to apply Lemma \ref{lemma1} to the $m$-fold amplification, so that the constant $C_p^{1/m}$ converges to $1$ in the limit. An alternative proof can be obtained by reducing to the setting of \cite[Theorem 5.5.1]{Gra}.

\begin{prop} \label{prop:isometry}
Take $p\in[1, \infty)$, a Hilbert space $\mathcal{H}$ and consider representations $\lambda^p\colon G\rightarrow  \mbox{Isom}( \ell^p(G))$ and $\lambda^{p,\mc{H}}\colon  G\rightarrow \mbox{Isom}(\ell^p(G;\mc{H}))$ given by the left translation. Then for every  $f\in \ell^1(G)$ we have
\[
\|\lambda^p(f)\|_{\mc{B}(\ell^p(G))}=
\|\lambda^{p,\mc{H}}(f)\|_{\mc{B}(\ell^p(G;\mc{H}))}.
\]
\end{prop}

\begin{proof}
We have established the inequality $\|\lambda^p(f)\|_{\mc{B}(\ell^p(G))}\le
\|\lambda^{p,\mc{H}}(f)\|_{\mc{B}(\ell^p(G;\mc{H}))}$ as a part of Lemma \ref{lemma1} for $A=\C$. Consequently, it is enough to prove the converse inequality. Furthermore, arguing by approximation, we only need to establish this inequality for finitely supported functions.

Assume by contradiction, that we have $\bigl\|\lambda^p(f)\bigr\|_{\mc{B}(\ell^p(G))}< \bigl\|\lambda^{p,\mc{H}}(f) \bigr\|_{\mc{B}(\ell^p(G;\mathcal{H}))}$ for a finitely supported $f\in \ell^1(G)$. Then we can find $\eps>0$ and a finitely supported function $\xi\in \ell^p(G;\mc{H})$ with $\xi(g)=\sum_{n=1}^{N} \xi_{g,n} e_n\,(g\in G)$ for an orthonormal family $\{e_n\}_{n=1}^{N}\subseteq \mathcal{H}$ such that
\begin{equation}\label{eq1}
\|\xi\|=1,\quad (1+ \eps)
\bigl\|\lambda^p(f) \bigr\|_{\mc{B}(\ell^p(G))}\le \bigl\|\lambda^{p,\mc{H}}(f)\xi \bigr\|_{\ell^p(G;\mathcal{H})}.
\end{equation}
Fix $m\geq 2$. For simplicity, we shall write $\mathbf{n}=(n_1,\ldots,n_m)\in \{1,\ldots,N\}^m$ and $e_{\mathbf{n}}:=e_{n_1} \otimes \cdots \otimes  e_{n_m} \in \mathcal{H}^{\otimes m}$. Consider the vector
\[
\eta_m \colon G^{\times m}\ni  (g_1,\dotsc,g_m) \mapsto
\sum_{\mathbf{n} \in \{1,\ldots,N\}^m}
\xi_{g_1,n_1}\cdots \xi_{g_m,n_m}
e_{\mathbf{n}}\in \mc{H}^{\otimes m}
\]
in $\ell^p(G^{\times m};\mathcal{H}^{\otimes m})$. As easily checked, it has norm $1$:
\begin{align*}
\|\eta_m\|&=
\bigl(
\sum_{g_1,\dotsc,g_m\in G}
\bigl\|
\sum_{\mathbf{n} \in \{1,\ldots,N\}^m}
\xi_{g_1,n_1}\cdots \xi_{g_m,n_m}
e_{\mathbf{n}}
\bigr\|^{p}
\bigr)^{1/p}\\
&=
\bigl(
\sum_{g_1,\dotsc,g_m\in G}
\bigl(\sum_{n_1,\dotsc,n_m=1}^{N}
\bigl|
\xi_{g_1,n_1}\cdots \xi_{g_m,n_m}\bigr|^{2}\bigr)^{p/2}
\bigr)^{1/p}\\
&=
\bigl(
\sum_{g_1,\dotsc,g_m\in G}
\bigl(\sum_{n_1=1}^{N}
|\xi_{g_1,n_1}|^{2}\bigr)^{p/2}
\cdots
\bigl(
\sum_{n_m=1}^{N}
|\xi_{g_m,n_m}|^{2}
\bigr)^{p/2}
\bigr)^{1/p}\\
&=
\bigl(
\sum_{g_1\in G}
\bigl(\sum_{n_1=1}^{N}
|\xi_{g_1,n_1}|^{2}\bigr)^{p/2}
\bigr)^{1/p}
\cdots\bigl(\sum_{g_m\in G}
\bigl(
\sum_{n_m=1}^{N}
|\xi_{g_m,n_m}|^{2}
\bigr)^{p/2}
\bigr)^{1/p}=1.
\end{align*}
Next, define
\[
x_m=\sum_{g_1,\dotsc,g_m\in G}
f(g_1)\cdots f(g_m)\lambda^{p,\mathcal{H}^{\otimes m}}(g_1,\dotsc,g_m)\in \mc{B}(\ell^p(G^{\times m};\mathcal{H}^{\otimes m})).
\]
where $\lambda^{p,\mc{H}^{\otimes m}}\colon G^{\times m}\rightarrow \mbox{Isom}(\ell^p(G^{\times m};\mc{H}^{\otimes m}))$ is the left translation representation.
We can use Lemma \ref{lemma1} for the group $G^{\times m}$ and Hilbert space $\mathcal{H}^{\otimes m}$, to get
\begin{equation}\label{eq2}
\|x_m\|\le C_p
 \bigl\|\lambda^p(f^{\otimes m})\bigr\|_{\mc{B} (\ell^p(G^{\times m}))}
=
C_p \|\lambda^p(f)\|_{\mc{B}(\ell^p(G))}^m.
\end{equation}
To see the above equality, observe that we have the canonical isometric isomorphism $\ell^p(G^{\times m})\simeq \ell^p(G)\otimes_{d_p}\ell^{p}(G^{\times (m-1)})$, where $d_p$ is the Chevet-Saphar tensor product (see \cite[Section 4]{Chevet} and \cite[Section 6.2]{Ryan}). Under this isomorphism $\lambda^p(f^{\otimes m})\simeq \lambda^p(f) \otimes \lambda^p(f^{\otimes (m-1)})$, hence the claim holds by induction. On the other hand
\begin{align*}
\|x_m\|&\geq \|x_m \eta_m\|\\
&=
\bigl\|
\bigl(
\sum_{h_1,\dotsc,h_m\in G}
f(h_1)\cdots f(h_m) \lambda^{p,\mathcal{H}^{\otimes m}} (h_1,\dotsc,h_m)
\bigr)
\eta_m\bigr\|\\
&=
\bigl(\sum_{g_1,\dotsc,g_m\in G}\bigl\|
\sum_{h_1,\dotsc,h_m\in G}
f(h_1)\cdots f(h_m)
\eta_m(h_1^{-1} g_1,\dotsc, h_m^{-1} g_m)
\bigr\|^p\bigr)^{1/p}\\
&=
\bigl(\sum_{g_1,\dotsc,g_m\in G}\bigl\|
\sum_{h_1,\dotsc,h_m\in G}
f(h_1)\cdots f(h_m)
\sum_{\mathbf{n}\in \{1,\dotsc,N\}^m}
\xi_{h_1^{-1} g_1, n_1}\cdots
\xi_{h_m^{-1} g_m, n_m} e_{\mathbf{n}}
\bigr\|^p\bigr)^{1/p}\\
&=
\bigl(\sum_{g_1,\dotsc,g_m\in G}
\bigl(
\sum_{\mathbf{n}\in \{1,\dotsc,N\}^m}
\bigl|
\sum_{h_1,\dotsc,h_m\in G}
f(h_1)\cdots f(h_m)
\xi_{h_1^{-1} g_1, n_1}\cdots
\xi_{h_m^{-1} g_m, n_m}
\bigr|^2\bigr)^{p/2}\bigr)^{1/p}\\
&=
\bigl(\sum_{g_1,\dotsc,g_m\in G}
\bigl(
\sum_{\mathbf{n}\in \{1,\dotsc,N\}^m}
\bigl|
\sum_{h_1\in G}
f(h_1)
\xi_{h_1^{-1} g_1, n_1}
\bigr|^2\cdots
\bigl|
\sum_{h_m\in G}
f(h_m)
\xi_{h_m^{-1} g_m, n_m}
\bigr|^2  \bigr)^{p/2}\bigr)^{1/p}\\
&=
\bigl(\sum_{g_1,\dotsc,g_m\in G}
\bigl(
\sum_{n_1=1}^N
\bigl|
\sum_{h_1\in G}
f(h_1)
\xi_{h_1^{-1} g_1, n_1}
\bigr|^2\bigr)^{p/2} \cdots
\bigl(
\sum_{n_m=1}^N
\bigl|
\sum_{h_m\in G}
f(h_m)
\xi_{h_m^{-1} g_m, n_m}
\bigr|^2\bigr)^{p/2}
\bigr)^{1/p}\\
&=
\bigl(
\sum_{g\in G}\bigl(
\sum_{n=1}^{N}
\bigl|
\sum_{h\in G} f(h)\xi_{h^{-1} g,n}
\bigr|^2
\bigr)^{p/2}
\bigr)^{m/p}=
\bigl(
\sum_{g\in G}
\bigl\|
\sum_{n=1}^{N}
\sum_{h\in G} f(h)\xi_{h^{-1} g,n} e_n
\bigr\|^p
\bigr)^{m/p}\\
&=
\bigl(
\sum_{g\in G}
\bigl\|
(\lambda^{p,\mc{H}}(f) \xi )(g)
\bigr\|^p
\bigr)^{m/p}
=
\bigl\|
\lambda^{p,\mc{H}}(f)\xi \bigr\|^m_{\ell^p(G;\mc{H})}
\\&\geq (1+\eps )^m \bigl\|\lambda^p(f) \bigr\|_{\mc{B}(\ell^p(G))}^m,
\end{align*}
where the last inequality holds by \eqref{eq1}. Combining this with \eqref{eq2} gives contradiction for $m$ big enough,  so that $(1+\eps)^m> C_p$.
\end{proof}

We immediately obtain a conclusion in the case of $A=\C$.

\begin{thm}\label{thm:coincide}
For every group $G$ and $1< p <\infty$ we have
	$\PF{p}{G}{\mathbb{C}} =\textup{PF}^*_p(G)$, where the latter denotes the Banach $*$-algebra introduced in \cite{SW2}.
\end{thm}
\begin{proof}
The definition of $\textup{PF}^*_p(G)$ in \cite{SW1} uses the norm
\begin{align*}
	\| \cdot \|_{\mathrm{id},p}\colon \ell^1(G)\to [0,\infty),\quad f \mapsto	 \mbox{max}\lbrace
	\|  \lambda^p(f)\|_{\mathcal{B}(\ell^p(G))},\,
	\|  \lambda^q(f)\|_{\mathcal{B}(\ell^q(G))}\rbrace,
\end{align*}
whereas the norm in $\PF{p}{G}{\mathbb{C}}$ is given by
\begin{align*}
	\| \cdot \|_{\PF{p}{G}{\C}}\colon \ell^1(G)&\to [0,\infty), \\
& f \mapsto	\sup_{\mathcal{H} - \textup{Hilbert space}} \mbox{max}\lbrace
	\|  \lambda^{p,\mc{H}}(f) \|_{\mathcal{B}(\ell^p(G ; \mathcal{H}))},\,
	\|  \lambda^{q,\mc{H}}(f)\|_{\mathcal{B}(\ell^q(G ; \mathcal{H}))}\rbrace,
\end{align*}
where $\lambda^{p,\mc{H}}$ is the left translation representation of $G$ on $\ell^p(G;\mc{H})$. The independence of the norm of the choice of $\mathcal{H}$  is precisely the content of Proposition \ref{prop:isometry}.	
\end{proof}

Theorem \ref{thm:coincide} can be interpreted by saying that the Banach $*$-algebra $\PF{p}{G}{\C}$ is (isometrically) independent from the choice of representation of $\C$. For a general $G$-C$^*$-algebra $A$ we have a slightly weaker statement -- $\PF{p}{G}{A}$ is independent from such a choice up to a (not necessarily isometric) isomorphism.

\begin{prop}\label{prop1}
Let $1\le p<\infty$, $\pi,\sigma$ be non-degenerate $*$-representations of $A$ on respectively $\mc{H}_\pi,\mc{H}_\sigma$. Assume that $\pi$ is faithful and take $f\in\ell^1(G;A)$. Then
\begin{itemize}
\item[(i)] if $\pi\simeq \aleph_0\cdot \pi$, i.e.~$\pi$ is unitarily equivalent with $\aleph_0$-many copies of $\pi$, we have
\begin{equation}\label{eq7}
\| (\wt{\sigma}_p\rtimes \lambda^p)(f) \|_{\mc{B}(\ell^p(G;\mc{H}_\sigma))}\le
\|(\wt{\pi}_p\rtimes\lambda^p)(f)\|_{\mc{B}(\ell^p(G;\mc{H}_\pi))};
\end{equation}
\item[(ii)] in general, there is a constant $C_p\ge 1$ depending only on $p$, such that
\begin{equation}\label{eq8}
\| (\wt{\sigma}_p\rtimes \lambda^p)(f) \|_{\mc{B}(\ell^p(G;\mc{H}_\sigma))}\le C_p
\|(\wt{\pi}_p\rtimes\lambda^p)(f)\|_{\mc{B}(\ell^p(G;\mc{H}_\pi))}.
\end{equation}
\end{itemize}
\end{prop}

\begin{proof}
Assume first that $\aleph_0\cdot \pi\simeq \pi$. Since $\wt{\sigma}_p\rtimes \lambda^p, \wt{\pi}_p\rtimes \lambda^p$ are contractive, it is enough to show \eqref{eq7} for a fixed finitely supported $f\colon G\rightarrow A$. Next, choose $\xi\in \ell^p(G;\mc{H}_\sigma)$ of norm $\le 1$ with finite support. To avoid trivialities, assume $f\neq 0, \xi\neq 0$. It is enough to show that
\begin{equation}\label{eq6}
\| (\wt{\sigma}_p\rtimes \lambda^p)(f) \xi\|_{\ell^p(G;\mc{H}_\sigma)}\le
\|(\wt{\pi}_p\rtimes\lambda^p)(f)\|_{\mc{B}(\ell^p(G;\mc{H}_\pi))}.
\end{equation}
We want to use Voiculescu's theorem on approximate unitary equivalence of representations, hence we need to make several preparations and reduce to a separable situation.

Let $G_0$ be the (countable) subgroup of $G$ generated by $\supp(f)\cup\supp(\xi)$. Since $\sigma$ is non-degenerate, there is $b_{g,r}\in A$ such that $\|\sigma(b_{g,r})\xi(g) - \xi(g)\|\le \tfrac{1}{r}$ for $r\in\N, g\in G_0$. Define a C$^*$-algebra
\[
B=C^*(\{\alpha_{h}(f(g)),\alpha_{h}(b_{g,r})\mid g,h \in G_0,r\in\N\})\subseteq A.
\]

Observe that $B$ is $G_0$-invariant and separable; choose a countable dense subset $\{a_k\}_{k\in\N}\!\subseteq B\setminus\{0\}$. Next, define a separable, closed subspace of $\mc{H}_\sigma$:
\[\mc{H}_{\sigma,0}=\ov{\lin}\{  \sigma(a_k)\xi(g)\mid g\in G_0, k\in\N\}.\]
Clearly $\mc{H}_{\sigma,0}$ is $\sigma(B)$-invariant, hence we can consider the (co-)restricted representation $\sigma_0\colon B\rightarrow \mc{B}(\mc{H}_{\sigma,0})$.

 As $\pi$ is faithful and $\aleph_0\cdot \pi\simeq \pi$, there are unit vectors $\{\zeta_{k,m}^i\}_{i,k,m\in\N}\subseteq \mc{H}_\pi$ such that $\|\pi(a_k)\zeta^i_{k,m}\| \ge \|a_k\|-\tfrac{1}{m}$ and $\la \pi(b)\zeta^i_{k,m} | \pi(b')\zeta^{i'}_{k',m'}\ra=0$, for any $b,b'\in A,\, i,k,m,k',m',i'\in \N$ such that $i\neq i'$. Finally, define a separable, closed subspace
\[\mc{H}_{\pi,0}=\ov{\lin} \{\pi(a_n)\zeta_{k,m}^i\mid i,n,k,m\in \N\}\subseteq \mc{H}_{\pi}.
\]
Then $\mc{H}_{\pi,0}$ is a $\pi(B)$-invariant subspace. Let $\pi_0$ be the (co-)restriction $\pi_0\colon B\rightarrow \mc{B}(\mc{H}_{\pi,0})$.\\

Representation $\pi_0$ is faithful. Indeed, for $a\in B$ we can find indices $n_j$ such that $a_{n_j}\xrightarrow[j\to\infty]{} a$. Then if $(c_k)_{k\in\N}$ is an approximate unit in $B$, choose $m_k\in\N$ such that $\|a_{m_k} - c_k\|\le\tfrac{1}{k}$ with $\|c_k\|,\|a_{m_k}\|\le 1$. Then for any $j\in \N$ we have
\begin{align*}
 \|\pi_0(a_{n_j})\|&\ge \liminf_{k\to\infty}
 \|\pi_0(a_{n_j}) \pi(a_{m_k})\zeta^1_{n_j,k}\|\ge
 \liminf_{k\to\infty} \bigl(
  \|\pi(a_{n_j} c_k)\zeta^1_{n_j,k}\|-\|a_{n_j}\|\tfrac{1}{k}
  \bigr)\\
  &\ge
   \liminf_{k\to\infty} \bigl(
  \|\pi(a_{n_j} )\zeta^1_{n_j,k}\|-
\|a_{n_j}-a_{n_j}c_k\|-
  \|a_{n_j}\|\tfrac{1}{k}
  \bigr) \ge \|a_{n_j}\|
\end{align*}
 and we obtain $\|\pi_0(a_{n_j})\|=\|a_{n_j}\|$. Finally
\[
\|a\|\ge \|\pi_0(a)\|=
\lim_{j\to\infty} \|\pi_0(a_{n_j})\|
=\lim_{j\to\infty} \|a_{n_j}\|=\|a\|.
\]
Clearly both $\pi_0,\sigma_0$ are non-degenerate. Since $b_{g,r}\in B$, we have $\xi\colon G\rightarrow \mc{H}_{\sigma,0}$.

Finally, we claim that $\pi_0(B)\cap \mc{K}(\mc{H}_{\pi,0})=\{0\}$. Indeed, suppose otherwise. Then there is $a\in B, \|a\|=1$ such that $\pi_0(a)$ is of finite rank. Then we can choose index $n\in \N$ so that $\|a-a_n\|\le \tfrac{1}{4}$, hence $\|a_n\|\ge \tfrac{3}{4}$. For any $i\in \N$ we have $\|\pi_0(a_n)\zeta^{i}_{n,4}\| \ge  \tfrac{1}{2}$, hence $\tfrac{1}{2} \le \|\pi_0(a_n)\zeta^{i}_{n,4}\| \le \|\pi_0(a)\zeta^i_{n,4}\|+ \tfrac{1}{4}$ and $\pi_0(a)\zeta^{i}_{n,4}\neq 0$. Since these vectors are pairwise orthogonals for different $i$'s, the operator $\pi_0(a)\in \mc{B}(\mc{H}_{\pi,0})$ cannot be of finite rank and we obtain a contradiction.\\

Consider the C$^*$-algebra $\pi_0(B)\subseteq \mc{B}(\mc{H}_{\pi,0})$. Since $\pi_0$ is faithful, we can consider a non-degenerate $*$-homomorphism $\sigma'\colon \pi_0(B)\ni \pi_0(a)\mapsto \sigma_0(a)\in \mc{B}(\mc{H}_{\sigma,0})$. Voiculescu's theorem \cite[Corollary II.5.5]{Davidson} gives us an approximate unitary equivalence $\mathrm{id}\approx_a \mathrm{id} \oplus \sigma'$ of representations of $\pi_0(B)$. This means that there is a sequence of unitary operators $U_n\colon \mc{H}_{\pi,0}\rightarrow \mc{H}_{\pi,0}\oplus \mc{H}_{\sigma,0}$ $(n\in\N)$ such that
\[
\lim_{n\to\infty} \| U_n \pi_0(a) U_n^* - \pi_0(a)\oplus \sigma_0(a)\|_{\mc{B}(\mc{H}_{\pi,0}\oplus\mc{H}_{\sigma,0})}=0\qquad(a\in B).
\]

With these preliminaries, we can start calculating the left hand side of \eqref{eq6}:
\begin{align*}
\|(\wt{\sigma}_p\rtimes \lambda^p)&(f)\xi\|^p_{\ell^p(G;\mc{H}_\sigma)}=
\sum_{h\in G_0}\bigl\|
\sum_{g\in G_0}\sigma (\alpha_{h^{-1}}(f(g)))\xi(g^{-1}h)\bigr\|_{\mc{H}_\sigma}^p\\
&=
\sum_{h\in G_0}\bigl\|
\sum_{g\in G_0}\bigl(0,\sigma_0 (\alpha_{h^{-1}}(f(g)))\xi(g^{-1}h) \bigr)\bigr\|_{\mc{H}_{\pi,0}\oplus \mc{H}_{\sigma,0}}^p\\
&=
\sum_{h\in G_0}\bigl\|
\sum_{g\in G_0}
\bigl(\pi_0(\alpha_{h^{-1}}(f(g)))\oplus \sigma_0(\alpha_{h^{-1}}(f(g))) \bigr)\,
\bigl(0,\xi(g^{-1}h) \bigr)\bigr\|_{\mc{H}_{\pi,0}\oplus \mc{H}_{\sigma,0}}^p\\
&=
\sum_{h\in G_0}\bigl\|
\sum_{g\in G_0}
\lim_{n\to\infty} U_n
\pi_0(\alpha_{h^{-1}}(f(g))) U_n^*
\bigl(0,\xi(g^{-1}h) \bigr)\bigr\|_{\mc{H}_{\pi,0}\oplus \mc{H}_{\sigma,0}}^p\\
&=
\sum_{h\in G_0}\lim_{n\to\infty} \bigl\|
U_n \sum_{g\in G_0}
\pi_0(\alpha_{h^{-1}}(f(g))) U_n^*
\bigl(0,\xi(g^{-1}h) \bigr)\bigr\|_{\mc{H}_{\pi,0}\oplus \mc{H}_{\sigma,0}}^p\\
&\le
\liminf_{n\to\infty}
\sum_{h\in G_0} \bigl\|
 \sum_{g\in G_0}
\pi_0(\alpha_{h^{-1}}(f(g))) U_n^*
\bigl(0,\xi(g^{-1}h) \bigr)\bigr\|_{\mc{H}_{\pi,0}}^p.
\end{align*}
Note that the sum over $g$ is finite and we use Fatou lemma in the last step. Now, define
\[
\eta_n\colon G\ni h\mapsto U_n^*(0,\xi(h))\in \mc{H}_{\pi,0}\subseteq \mc{H}_\pi.
\]
Clearly $\|\eta_n\|_{\ell^p(G;\mc{H}_{\pi,0})}\le 1$. We can continue the computation
\begin{align*}
\|(\wt{\sigma}_p\rtimes\lambda^p)(f)\xi\|^p_{\ell^p(G;\mc{H}_\sigma)}&\le
\liminf_{n\to\infty} \sum_{h\in G_0}\bigl\|
\sum_{g\in G_0}\pi_0(\alpha_{h^{-1}}(f(g))) \eta_n(g^{-1} h)\bigr\|^p_{\mc{H}_{\pi,0}}\\
&=
\liminf_{n\to\infty} \sum_{h\in G}\bigl\|
\sum_{g\in G}(\wt{\pi}_p(f(g)) \lambda^p(g)\eta_n)( h)\bigr\|^p_{\mc{H}_\pi}\\
&=
\liminf_{n\to\infty} \sum_{h\in G}\bigl\|
\bigl(
(\wt{\pi}_p\rtimes\lambda^p)(f) \eta_n \bigr)(h)\bigr\|^p_{\mc{H}_\pi}\\
&=
\liminf_{n\to\infty}
\|(\wt{\pi}_p\rtimes\lambda^p)(f)\eta_n\|_{\ell^p(G;\mc{H}_\pi)}^p\le
\|(\wt{\pi}_p\rtimes\lambda^p)(f)\|^p_{\mc{B}(\ell^p(G;\mc{H}_\pi))}.
\end{align*}
This establishes \eqref{eq6} and ends the proof of \eqref{eq7}.
\smallskip
The inequality \eqref{eq8} follows by considering amplification. Consider the representation $\gamma=\pi(\cdot)\otimes I_{\ell^2}\colon A\rightarrow \mc{B}(\mc{H}\otimes \ell^2)$. It is faithful, non-degenerate and $\gamma\simeq \aleph_0\cdot \gamma$. Using the inequality \eqref{eq7} and Lemma \ref{lemma1} we obtain
\[
\| (\wt{\sigma}_p\rtimes \lambda^p)(f) \|_{\mc{B}(\ell^p(G;\mc{H}_\sigma))}\le
\|(\wt{\gamma}_p\rtimes\lambda^p)(f)\|_{\mc{B}(\ell^p(G;\mc{H}_\gamma))}\le
C_p \|(\wt{\pi}_p\rtimes\lambda^p)(f)\|_{\mc{B}(\ell^p(G;\mc{H}_\pi))}
\]
for $f\in \ell^1(G;A)$, as originally claimed.
\end{proof}

\begin{cor}\label{cor:indepBanach}
Take $1<p<\infty$ and let $\pi,\sigma$ be faithful representations of a $G$-C*-algebra $A$. There exists $D_p\ge 1$ (depending only on $p$) with the following properties:
 \begin{itemize}
 \item[(i)] the identity map on $\ell^1(G;A)$ extends to an isomorphism of Banach $*$-algebras $\PF{p,\pi}{G}{A} \rightarrow \PF{p,\sigma}{G}{A}$ with the norm bounded by $D_p$. If $\aleph_0\cdot \pi\simeq \pi$ and $\aleph_0\cdot\sigma\simeq \sigma$, then this isomorphism is isometric.
 \item[(ii)] The identity map on $\ell^1(G;A)$ extends to an isomorphism of Banach $*$-algebras $\PF{p,\pi}{G}{A}\rightarrow \PF{p}{G}{A}$ with the norm bounded by $D_p$.
\end{itemize}
\end{cor}

\begin{proof}
The two claims follow immediately from Proposition \ref{prop1}, with $D_p=\max(C_p,C_q)$, where $q$ is the H{\"o}lder conjugate of $p$.
\end{proof}

\subsection{Further properties of the algebras $\PF{p}{G}{A}$}

We now investigate certain spectral and hereditary properties of Banach algebras $\textup{PF}^*_p(G;A)$. We start with the following result, which is a consequence of the Riesz-Thorin Theorem.

\begin{prop}\label{CP:cov:interpolation}
	Let $p_0, p_1 \in (1,\infty)$, $p_0< p_1$, $\theta \in (0,1)$, let $A$ be a $G$-C*-algebra and let
	$\pi\colon A \to \mathcal{B}(\mathcal{H})$ be a $*$-representation of $A$ on a
	Hilbert space $\mathcal{H}$. Let
	$p_\theta \in (p_0,p_1)$ be such that $1/p_\theta = (1-\theta)/p_0 + \theta/p_1$.
	Then for every $f\in \ell^1(G;A)$ we have
	\begin{align*}
		\| \widetilde{\pi}\rtimes \lambda^{p_\theta}(f)\|
		\leq \|\widetilde{\pi} \rtimes \lambda^{p_0}(f)\|^{1-\theta}
		\| \widetilde{\pi}\rtimes \lambda^{p_1}
		(f)\|^\theta
	\end{align*}
\end{prop}
\begin{proof}
	Let $f\in \ell^1(G;A)$. For every $g\in \ell^{p_0}(G;\mathcal{H})
	\cap \ell^{p_1}(G;\mathcal{H})$ the equality
	$\widetilde{\pi}\rtimes \lambda^{p_0}(f)g
	= \widetilde{\pi}\rtimes \lambda^{p_1}(f)g$ holds. Therefore
	\begin{align*}
		Z\colon \ell^{p_0}(G;\mathcal{H})+ \ell^{p_1}(G;\mathcal{H}) \to \ell^0(G;
		\mathcal{H}),\quad g_0+g_1 \mapsto \widetilde{\pi}\rtimes \lambda^{p_0}(f)g_0
		+ \widetilde{\pi}\rtimes \lambda^{p_1}(f)g_1
	\end{align*}
	is a well-defined linear map. By the Riesz-Thorin Theorem, we obtain that
	\begin{align*}
		\|\widetilde{\pi}\rtimes \lambda^{p_\theta}(f) g\|
		\leq \| \widetilde{\pi}\rtimes \lambda^{p_0}(f)\|^{1-\theta}
		\| \widetilde{\pi}\rtimes \lambda^{p_1}(f)\|^\theta \| g\|_{{p_\theta}}
	\end{align*}
	for all $g\in \ell^{p_\theta}(G;\mathcal{H})$. In particular,
	\begin{align}
		\|\widetilde{\pi}\rtimes \lambda^{p_\theta}(f)\| \leq
		\| \widetilde{\pi}\rtimes \lambda^{p_0}(f)\|^{1-\theta} \|\widetilde{\pi}
		\rtimes \lambda^{p_1}(f)\|^\theta
	\end{align}
	holds.
\end{proof}

\begin{prop}\label{CP:cov:interpolationPseodo}
	Let $p_0, p_1 \in (1,\infty)$, $p_0 < p_1$ and $ p \in (p_0,p_1)$. Let
	$\theta \in (0,1)$ be such that $1/p = (1-\theta)/p_0 + \theta/p_1$ and let $A$ be a $G$-C*-algebra. Then
	\begin{align*}
		\|f \|_{\textup{PF}^*_p(G;A)}\leq \| f \|_{\textup{PF}^*_{p_0}(G;A)}^{1-\theta}\|f\|_{\textup{PF}_{p_1}^*(G;A)}^\theta
	\end{align*}
	for all $f\in \ell^1(G;A)$. %In particular, if $1<p_0<p\leq 2$, then
    %\begin{align*}
	%	\|f \|_{\textup{PF}^*_{p}(G;A)}\leq \| f \|_{\textup{PF}^*_{p_0}(G;A)}
	%\end{align*}
    %	for all $f\in \ell^1(G;A)$.
\end{prop}

\begin{proof}
	Let $\pi\colon A \to \mathcal{B}(\mathcal{H})$ be any non-degenerate $*$-representation of $A$ on a
	Hilbert space $\mathcal{H}$. Let $q_i\in (1,\infty)$ be such that $1/p_i + 1/q_i =1$, for $i \in \{ 0 ,1 \}$ and
	let $q$  be such that $1/q = (1-\theta) /q_0 + \theta /q_1$. Then
	$1/p + 1/q = (1-\theta)(1/p_0 + 1/q_0)+ \theta(1/p_1+1/q_1) = 1$. By
	Proposition \ref{CP:cov:interpolation}, for each  $f\in \ell^1(G;A)$ we have
	\[
		\|\tilde{\pi}\rtimes \lambda^p(f)\| \leq \|\tilde{\pi}\rtimes \lambda^{p_0}(f) \|^{1-\theta}
		 \|\tilde{\pi}\rtimes \lambda^{p_1} (f)\|^{\theta} \] and \[
		\|\tilde{\pi}\rtimes \lambda^q(f)\| \leq \|\tilde{\pi}\rtimes \lambda^{q_0}(f) \|^{1-\theta}
		 \|\tilde{\pi}\rtimes \lambda^{q_1} (f)\|^{\theta}.
\]
Thus
	\begin{align*}
		\| f \|_{\pi,p} \leq \|f\|_{\pi,p_0}^{1-\theta} \|f\|_{\pi,p_1}^{\theta}
	\end{align*}
and finally
	\begin{align*}
		\|f\|_{\textup{PF}^*_p(G;A)}\leq \|f\|_{\textup{PF}^*_{p_0}(G;A)}^{1-\theta}\|f\|_{\textup{PF}^*_{p_1}(G;A)}^\theta.
	\end{align*}
\end{proof}

The last proposition yields, in particular, the following monotonicity result.

\begin{cor} \label{prop:monotone}
  Let $p, p' \in (1,2]$, $p' >p$, and let $A$ be a $G$-C*-algebra. Then, for every $f \in \ell^1(G;A)$, we have
  \[\|f\|_{\textup{PF}^*_{p'}(G;A)} \leq \|f\|_{\textup{PF}^*_{p}(G;A)}.\]
  In particular, the identity map on $\ell^1(G;A)$ extends to a contractive $*$-homomorphism from
  $\textup{PF}^*_{p}(G;A)$ into $\textup{PF}^*_{p'}(G;A)$.
\end{cor}
\begin{proof}
This can be shown as in \cite[Proposition 4.5]{SW1}. Indeed, if we let $q> 2$ be the conjugate of $p$ and choose $\theta\in(0,1)$ with $1/p'=(1-\theta)/p+\theta/q$, then Proposition \ref{CP:cov:interpolationPseodo} shows that
\begin{align*}
   		\|f \|_{\textup{PF}^*_{p'}(G;A)} \leq \| f \|_{\textup{PF}^*_{p}(G;A)}^{1-\theta}\|f\|_{\textup{PF}_{q}^*(G;A)}^\theta=\| f \|_{\textup{PF}^*_{p}(G;A)},
\end{align*}
as $\| f \|_{\textup{PF}^*_{p}(G;A)}=\| f \|_{\textup{PF}^*_{q}(G;A)}$.
%\textcolor{blue}{}
%{\color{red} Do we want to add the details?\textcolor{blue}{Done!}}.
\end{proof}

\begin{prop}\label{cp:cp:nestedAlg}
	Let $A$ be a $G$-C*-algebra and $p\in (1,\infty)$. Then
    $\textup{PF}^*_p(G;A)$ is a $*$-subalgebra of $A\rtimes_r G$. In particular, $\textup{PF}^*_p(G;A)$ is a semisimple Banach $*$-algebra.
\end{prop}

\begin{proof}
By Remark \ref{rem:contractive} we may assume that $p \in (1,2)$.
It then follows from Corollary \ref{prop:monotone} that there is a contractive $*$-homomorphism
	$j_{A,p}\colon \textup{PF}_p^*(G;A) \to A\rtimes_r G$
	such that
	\begin{equation*}
		\xymatrix{
			& \ell^1(G;A) \ar[ld]^-{\iota_{A,p}} \ar[rd]^-{\iota_{A,2}} & \\
			\textup{PF}^*_p(G;A)\ar[rr]^-{j_{A,p}} & & A\rtimes_r G
		}
	\end{equation*}
	commutes, where $\iota_{A,2}\colon \ell^1(G;A)\to A\rtimes_r G$ and $\iota_{A,p}\colon  \ell^1(G;A) \to \textup{PF}_p^*(G;A)$ are the corresponding canonical contractive inclusions (see Remark \ref{rem:contractive}). In particular, $\iota_{A,2}$ and $\iota_{A,p}$ are injective. It remains to show that $j_{A,p}$ is injective to complete the 	proof. Let $x\in \ker j_{A,p}$. Since $\iota_{A,p}$ has dense image and $C_c(G;A)$ is dense in $\ell^1(G;A)$, there is a sequence $(f_n)_{n\in \N}\in C_c(G;A)^\mathbb{N}$ such that $\iota_{A,p}(f_n) \to x$ as $n\to \infty$.	
	It is enough to show that for all $*$-representations $\pi\colon A \to \mathcal{B}(\mathcal{H})$, we have
\begin{equation}\label{eq10}
	\lim_{n\to\infty}\tilde{\pi}\rtimes \lambda^p(f_n)=0\quad\textnormal{and}\quad \lim_{n\to\infty}\tilde{\pi}\rtimes \lambda^q(f_n) =0.
	\end{equation}
So let $\pi\colon A \to \mathcal{B}(\mathcal{H})$ be a $*$-representation of $A$ on a Hilbert space $\mathcal{H}$. Then, for all $g\in C_c(G;\mathcal{H})$, we have $\tilde{\pi}\rtimes \lambda^2(f_n)g \to 0$ in $\ell^2(G;\mathcal{H})$ as $n\to \infty$. Therefore, $\tilde{\pi}\rtimes \lambda^2(f_n)g \to 0$
	pointwise (with respect to elements of $G$) as $n\to \infty$ for all $g\in C_c(G;\mathcal{H})$. Note that we have
	\begin{align}\label{cp:cp:eq:regrep}
		\tilde{\pi}\rtimes \lambda^p(f_n)g = \tilde{\pi}\rtimes \lambda^2(f_n) g
		= \tilde{\pi}\rtimes \lambda^q(f_n) g
	\end{align}
	for all $g\in C_c(G;\mathcal{H})$.
	
	Since $(\iota_{A,p}(f_n))_{n\in\N}$ is a Cauchy sequence in $\PF{p}{G}{A}$ both (norm) limits in \eqref{eq10} exist. In particular, for every  $g\in C_c(G;\mathcal{H})$ the sequences
	$(\tilde{\pi}\rtimes \lambda^p(f_n) g)_{n\in\N}$ and $(\tilde{\pi}\rtimes \lambda^q(f_n)g)_{n\in\N}$ converge in $\ell^p(G;\mathcal{H})$ and, respectively, in $\ell^q(G;\mathcal{H})$. It follows that they converge pointwise. Therefore, by equation (\ref{cp:cp:eq:regrep}) and
	the fact that $\tilde{\pi}\rtimes \lambda^2(f_n) g\to 0$ as $n\to\infty$, we obtain
\[
\lim_{n\to\infty}
\tilde{\pi}\rtimes \lambda^p(f_n) g=0,\quad
\lim_{n\to\infty}\tilde{\pi}\rtimes \lambda^q(f_n) g=0
\]
for all $g\in C_c(G;\mc{H})$. Equation \eqref{eq10} follows.
\end{proof}

Next we have the following observation.

\begin{prop}\label{prop:approximate}
	Let $A$ be a $G$-C*-algebra and let $p\in(1,\infty)$.  The Banach $*$-algebra
	$\textup{PF}_p^*(G;A)$ admits a contractive approximate identity.
\end{prop}
\begin{proof}
The contractive inclusion
	$\iota_{A,p}\colon \ell^1(G;A) \to \textup{PF}_p^*(G;A)$ has dense image. The Banach $*$-algebra
	$\ell^1(G;A)$ admits a contractive approximate identity $(u_\lambda)_{\lambda\in \Lambda}
 $. Hence $(\iota_{A,p}(u_\lambda))_{\lambda\in \Lambda}$
	is a contractive net. So let $x\in \textup{PF}_p^*(G;A)$ and let
	$\varepsilon>0$. Then there is $f\in \ell^1(G;A)$ such that
	$\| x - \iota_{A,p}(f)\|< \varepsilon$. Since $(u_\lambda)_{\lambda\in \Lambda}$
	is an approximate identity in $\ell^1(G;A)$, there is a $\lambda_0 \in \Lambda$
	such that $\| u_\lambda f - f\|_1 <\varepsilon$ for all $\lambda \geq \lambda_0$. So
	we obtain (using the facts that both $(\iota_{A,p}(u_\lambda))_{\lambda\in \Lambda}$ and $\iota_{A,p}$
	are contractive)
	\begin{align*}
		\| \iota_{A,p}(u_\lambda) x-x \| &\leq \|\iota_{A,p}(u_\lambda)x - \iota_{A,p}(u_\lambda)
		\iota_{A,p}(f) \|+ \| \iota_{A,p}(u_\lambda f) - \iota_{A,p}(f)\|
		+\| \iota_{A,p}(f)- x \|\\
		&\leq\|x - \iota_{A,p}(f)\| + \|u_\lambda f -f \|_1 + \|\iota_{A,p}(f) -x\| \leq 3\varepsilon
	\end{align*}
	for all $\lambda \geq \lambda_0$. This shows that $(\iota_{A,p}(u_\lambda))_{\lambda\in \Lambda}$
	is a contractive left approximate identity in $\textup{PF}_p^*(G;A)$. Analogously it can be shown that
	$(\iota_{A,p}(u_\lambda))_{\lambda\in \Lambda}$ is a right approximate identity. This completes the
	proof.
\end{proof}

We shall now discuss certain functorial properties of our construction.

\begin{prop}\label{P:vector value extension of homomorphism}
	Let $A$ and $B$ be $G$-C*-algebras, let $p\in (1,\infty)$ and let $\varphi\colon A \to B$ be a $G$-equivariant
	$*$-homomorphism. Recall the maps $\iota_{A,p}$ defined in Remark \ref{rem:contractive} and $j_{A,p}$ introduced in the proof of Proposition \ref{cp:cp:nestedAlg}. Then there is a unique contractive $*$-homomorphism
	$$\varphi \rtimes_{\textup{PF}_p} G\colon  \textup{PF}_p^*(G;A)\to \textup{PF}_p^*(G;B)$$ such that
	\begin{equation}
		\xymatrix{
			\ell^1(G; A) \ar[d]^-{\iota_{A,p}} \ar[r]^-{\varphi\rtimes_1 G}&
			\ell^1(G;B)\ar[d]^-{\iota_{B,p}}\\
			\PF{p}{G}{A} \ar[r]^-{\varphi \cpf} \ar[d]^-{j_{A,p}}
			&\PF{p}{G}{B}\ar[d]^-{j_{B,p}}\\
			A\rtimes_r G \ar[r]^-{\varphi\rtimes_r G}& B\rtimes_r G
		}
	\end{equation}
	commutes.
\end{prop}
\begin{proof}
	The only thing we need to show is that $\varphi\rtimes_1 G$ is continuous with respect to
	the norms $\|\cdot\|_{\PF{p}{G}{A}}$ and $\|\cdot\|_{\PF{p}{G}{B}}$.
	But as the estimate
	\begin{align*}
		&\|(\varphi\rtimes_1 G)(f)\|_{\PF{p}{G}{B}}\\
		&=\sup\lbrace \max\lbrace \| \tilde{\pi}\rtimes \lambda^p((\varphi \rtimes_1 G)(f))\|,\;
		\| \tilde{\pi}\rtimes \lambda^q((\varphi\rtimes_1 G)(f))\|\rbrace \mid\\
&
\qquad\qquad\qquad		
\qquad\qquad\qquad
\qquad\qquad\qquad
\qquad\qquad\qquad
\mid
 \pi\ \mbox{is a} \ast\text{-representation
			of } B \, \rbrace\\
		&=
		\sup\lbrace \max\lbrace \| \widetilde{\pi\circ \varphi}\rtimes \lambda^p(f)\|,\,
		\| \widetilde{\pi\circ\varphi}\rtimes \lambda^q(f)\|\rbrace \mid \pi\ \mbox{is a} \ast\text{-representation
			of } B \, \rbrace\\
		&\leq
		\sup\lbrace \max\lbrace \| \widetilde{\pi}\rtimes \lambda^p(f)\|,\,
		\| \widetilde{\pi}\rtimes \lambda^q(f)\|\rbrace \mid \pi\ \mbox{is a} \ast\text{-representation
			of } A  \, \rbrace\\
		&= \|f \|_{\PF{p}{G}{A}},
	\end{align*}
	is valid for all $f\in \ell^1(G;A)$, we have that $\varphi \rtimes_1 G$ is contractive
	with respect to the $\mathrm{PF}_p^*$-norms. This completes the proof.
\end{proof}

We shall now check that the above morphism remains an isometry when we are dealing with  inclusions. This essentially follows from the fact that we can always extend representations from  C*-subalgebras, if we allow ourselves to increase the dimension of the representation space.

\begin{prop}\label{cp:cp:hereAlgPf}
	Let $p\in (1,\infty)$, $A$ be a $G$-C*-algebra and let $B$ be a $G$-invariant
 C*-subalgebra of $A$.
	Then the inclusion map $\iota\colon B \to A$ induces an isometric $*$-monomorphism
	$\iota \cpf\colon \textup{PF}_p^*(G;B) \to \textup{PF}_p^*(G;A)$.
	% and $\textup{PF}_p^*(G;B)$ is a hereditary Banach $*$-subalgebra of $\textup{PF}_p^*(G;A)$.
\end{prop}

\begin{proof}
	Let $f\in \ell^1(G; B)$. We have to show that
	$\|(\iota\cpf)(f)\|_{\textup{PF}_p^*(G;A)} = \|f\|_{\textup{PF}_p^*(G;B)}$. Since the inequality
	$\|(\iota\cpf)(f)\|_{\textup{PF}_p^*(G;A)} \leq \|f\|_{\textup{PF}_p^*(G;B)}$ is automatic by Proposition \ref{P:vector value extension of homomorphism},
	it suffices to show that
	\[
	\|(\iota\cpf)(f)\|_{\textup{PF}_p^*(G;A)} \geq \|f\|_{\textup{PF}_p^*(G;B)}.\]
	In order to do this, let
	$\varepsilon>0$ be given.
	By the definition of the norm $\|\cdot\|_{\textup{PF}^*_p(G;B)}$, there is a non-degenerate $*$-representation $\pi\colon B\to \mathcal{B}(\mathcal{H})$
	such that
	$$\| f\|_{\textup{PF}^*_p(G;B)}< \max \lbrace \|\tilde{\pi}\rtimes \lambda^p(f)\|,
	\, \|\tilde{\pi}\rtimes \lambda^q(f)\|\rbrace + \varepsilon.$$
%	Now  there is a Hilbert space $\mc{K}$, such that $\mc{H} \subseteq \mc{K}$ and a  $*$-representation 	$\overline{\pi}\colon A\to \mathcal{B}(\mathcal{K})$ such that 	$\pi(\cdot) = (\overline{\pi}\circ \iota(\cdot))|_{\mc{H}}$ (see the argument in the proof of Lemma \ref{lem:condexp} below).
	
Now, there is a Hilbert space $\mc{K}$, an isometry $V\colon \mc{H}\rightarrow \mc{K}$ and a $*$-representation $\rho\colon A\to \mathcal{B}(\mathcal{K})$ such that
	$\pi(b) = V^*\rho(b) V$ for $b\in B$. Indeed, this follows by applying first the (possibly the non-unital version of) Arveson's extension theorem to $\pi$, then unitizing the resulting extension via Proposition 2.2.1 or Lemma 2.2.3 in \cite{BO} and then applying the Stinespring Theorem.

	Note that we have an isometric inclusion $V\otimes 1\colon \ell^p(G;\mc{H})\subseteq \ell^p(G;\mc{K})$ and a natural contractive projection $V^*\otimes 1\colon \ell^p(G;\mc{K}) \to \ell^p(G;\mc{H})$.
	Consequently
	\begin{align*}
&\quad\;		\| (\iota\cpf)(f)\|_{\textup{PF}_p^*(G;A)}+\varepsilon\\
 &\geq
		\max \lbrace \|\tilde{\rho}\rtimes \lambda^p((\iota\rtimes_1 G)(f))\|,\,
		\|\tilde{\rho}\rtimes \lambda^q((\iota\rtimes_1 G)(f))\| \rbrace+\varepsilon\\
		&= \max \lbrace\|\widetilde{\rho\circ\iota}\rtimes \lambda^p(f)\|,
		\|\widetilde{\rho\circ \iota}\rtimes \lambda^q(f)\| \rbrace + \varepsilon \\
		&\geq	\max \lbrace\|(V^*\otimes 1)(\widetilde{\rho\circ\iota}\rtimes \lambda^p(f)) (V\otimes 1)\|,
		\|(V^*\otimes 1)(\widetilde{\rho\circ \iota}\rtimes \lambda^q(f))(V\otimes 1)\| \rbrace + \varepsilon \\
		&=\max \lbrace\|\tilde{\pi}\rtimes \lambda^p(f)\|,
		\|\tilde{\pi}\rtimes \lambda^q(f)\| \rbrace + \varepsilon \geq \|f \|_{\textup{PF}_p^*(G;B)}.
	\end{align*}
	Since $\varepsilon$ was chosen arbitrary, we obtain
	\begin{align*}
		\|f \|_{\textup{PF}^*_p(G;B)}\leq \|(\iota\cpf)(f) \|_{\textup{PF}^*_p(G;A)}.
	\end{align*}
	It follows that $\iota \cpf$ is isometric.
\end{proof}

\section{Crossed product functor $\rtimes_{\textup{PF}_p} G$}  \label{sec:C*}

In this short, but central section of the paper, we formally introduce the $\rtimes_{\textup{PF}_p} G$ functor, exploiting Banach algebras constructed and studied in the previous section.

Recall that $\ell^1(G;\cdot)$ is a functor from
the category of $G$-C*-algebra into the category of $*$-semisimple
Banach $*$-algebras and $C^*(\cdot)$ is functor from the category of $*$-semisimple
Banach $*$-algebras into the category of C*-algebras. Recall also that the latter is constructed by defining, for any  $*$-semisimple
Banach $*$-algebra $B$, the new, in general smaller, norm
\[ \|b\|_{C^*(B)}:= \sup\{\|\gamma(b)\| \mid
\mc{H}\textup{ is a Hilbert space and } \gamma\colon B \to \mc{B}(\mc{H}) \textup{ a}
\ast\textup{-homomorphism}\}\]
and defining $C^*(B)$, called the enveloping C*-algebra of $B$, as the completion of $B$ with respect to $\|\cdot\|_{C^*(B)}$.
 The composition
$C^*\circ \ell^1(G;\cdot)$ is then exactly the  universal crossed product functor $-\rtimes_{u}G$. It follows from Lemma \ref{cp:cp:nestedAlg} and Proposition \ref{P:vector value extension of homomorphism} that $\PF{p}{G}{\cdot}$ is a functor from
the category of $G$-C*-algebra into the category of $*$-semisimple Banach $*$-algebras. Thus we can state the following definition.

\begin{defn}
   For $p\in (1,\infty)$, denote by $\rtimes_{\textup{PF}_p} G$  the functor $C^* \circ \textup{PF}^*_p(G;\cdot)$
	from the category of $G$-C*-algebra to the category of C*-algebras. In particular, for every $G$-C*-algebra $A$ write
$A\rtimes_{\textup{PF}_p} G$ for the enveloping C*-algebra 	$C^*(\PF{p}{G}{A})$.
\end{defn}

We can immediately record a consequence of Corollary \ref{cor:indepBanach}.

\begin{thm}\label{T:C* envlope PF cross product}
	Suppose that $A$ is a $G$-C*-algebra and $\pi\colon A \to \mc{B}(\mc{H})$ is a faithful $*$-representation of $A$. Set $C^*(\PF{p,\pi}{G}{A})$ to be the enveloping C*-algebra of $\PF{p,\pi}{G}{A}$. Then
	 the identity map on $\ell^1(G;A)$ extends to an isometric isomorphism of C$^*$-algebras
	\[
	C^*(\PF{p,\pi}{G}{A})\simeq
	C^*(\PF{p}{G}{A})=A\rtimes_{\mathrm{PF}_p} G.\]
\end{thm}
\begin{proof} Take $f\in \ell^1(G;A)$. Then $\|f\|_{C^*(\PF{p,\pi}{G}{A})}$ is the supremum of $\|\gamma(f)\|$, where $\gamma$ is a $*$-representation of $\PF{p,\pi}{G}{A}$ on a Hilbert space. By composing with the isomorphism from Corollary \ref{cor:indepBanach} (ii), we obtain a $*$-representation of $\PF{p}{G}{A}$. The same logic holds in the reversed direction, hence in the end we are looking at the supremum of the same family of numbers.
	\end{proof}

\begin{prop}
Let $p \in (1, \infty)$. The functor $\rtimes_{\mathrm{PF}_p} G$ is a crossed product functor.
\end{prop}

\begin{proof}
	Let $A$ be a $G$-C*-algebra. Then, by Lemma \ref{cp:cp:nestedAlg} and its proof, $\PF{p}{G}{A}$ is a $*$-semisimple Banach $*$-algebra and moreover the canonical contractive $*$-homomorphisms $\iota_{A,p}\colon \ell^1(G;A) \to \textup{PF}_p^*(G;A)$ and
	$j_{A,p}\colon \textup{PF}_p^*(G;A) \to A\rtimes_r G$  are injective and 	
	have dense images. Thus, we have surjective $*$-homomorphism
	$q_A = C^*(\iota_{A,p})\colon A\rtimes_{u} G \to A\rtimes_{\textup{PF}_p}G$ and
	$s_A = C^*(j_{A,p})\colon A\rtimes_{\textup{PF}_p}G \to A\rtimes_r G$.
\end{proof}

We will prove in Theorem \ref{thm:projectionprop} that $\rtimes_{\mathrm{PF}_p} G$ is a correspondence functor, by establishing its projection property. As a first step towards this goal, we prove that it has the ideal property.

\begin{prop}
	Let $p \in (1, \infty)$. The functor $\rtimes_{\mathrm{PF}_p}G$ has the ideal property.
\end{prop}
\begin{proof}
	Let $A$ be a $G$-C*-algebra, $I$ an ideal in $A$ invariant under the action of $G$ and let
	$\iota\colon I \to A$ be the canonical inclusion. By Proposition \ref{cp:cp:hereAlgPf}, the map $\iota \cpf \colon \textup{PF}_p^*(G;I)
	\to \textup{PF}^*_p(G;A)$ is an isometric $*$-monomorphism, hence $\textup{PF}_p^*(G;I)$ (identified with the image of $\iota\cpf$)	is a closed $*$-ideal in $\textup{PF}_p^*(G;A)$. The map $\iota \cpf$ extends to a $*$-homomorphism, denoted by the same symbol, $\iota \cpf\colon C^*(\textup{PF}_p^*(G;I))\to C^*(\textup{PF}_p^*(G;A))$. It remains to show that it is injective.

Choose $0\neq x\in C^*(\textup{PF}_p^*(G;I))$ and let $\Pi\colon C^*(\textup{PF}_p^*(G;I))\to \mc{B}(\mathcal{H})$ be a non-degenerate $*$-representation of $C^*(\textup{PF}_p^*(G;I))$ on a Hilbert space $\mathcal{H}$ such that $\Pi(x)\neq 0$. Let $\Pi_{\textup{PF}_p^*(G;I)}$ be the restriction of $\Pi$ to $\textup{PF}_p^*(G;I)$, it is also non-degenerate. Since by Proposition \ref{prop:approximate} the algebra $\textup{PF}_p^*(G;I)$ admits a contractive approximate identity, and $\PF{p}{G}{I}$ is an ideal in $\PF{p}{G}{A}$, $\Pi_{\textup{PF}_p^*(G;I)}$ extends to a $*$-homomorphism $\widetilde{\Pi}_0\colon  \textup{PF}_p^*(G;A)\rightarrow \mc{B}(\mathcal{H})$. Finally by the universal property of the enveloping C*-algebra, $\widetilde{\Pi}_0$ extends to a $*$-representation $\widetilde{\Pi}\colon C^*(\textup{PF}_p^*(G;A))\rightarrow \mc{B}(\mc{H})$. Write $x=\lim_{i\in I} x_i$ with $x_i\in \PF{p}{G}{I}$. We obtain
\[
\widetilde{\Pi}\bigl( (\iota\rtimes_{\textup{PF}_p}G)(x)\bigr)=
\lim_{i\in I}
\widetilde{\Pi}\bigl( (\iota\rtimes_{\textup{PF}_p}G)(x_i)\bigr)=
\lim_{i\in I}
\Pi_{\PF{p}{G}{I}}(x_i)=\Pi(x)\neq 0.
\]
This completes the proof.
\end{proof}

Next we prove one more auxiliary lemma.

\begin{lem} \label{lem:condexp}
	Let $p\in (1,\infty)$, $A$ be a $G$-C*-algebra and let $e\in \mathcal{M}(A)$ be a $G$-invariant projection. Then the conditional expectation $\mathbb{E}\colon A\to eAe,\, a \mapsto eae$ extends to a contractive linear map
	$\mathbb{E} \rtimes_1 G\colon \ell^1(G;A) \to \ell^1(G;eAe),\, f \mapsto \mathbb{E}\circ f$ and further
	to a contractive linear map $\mathbb{E}\cpf \colon  \PF{p}{G}{A} \to \PF{p}{G}{eAe}$ such that the diagram
	\begin{equation}
		\xymatrix{
			\ell^1(G;A) \ar[r]^-{\mathbb{E}\rtimes_1 G} \ar[d]_{\iota_{A,p}} & \ell^1(G;eAe) \ar[d]^-{\iota_{eAe,p}}\\
			\PF{p}{G}{A} \ar[r]^-{\mathbb{E}\cpf} & \PF{p}{G}{eAe}
		}
	\end{equation}
	commutes.
\end{lem}
\begin{proof}
	Let $\pi\colon eAe \to \mathcal{B}(\mathcal{H})$ be a non-degenerate $*$-homomorphism.
	Then, similarly as in the proof of Proposition \ref{cp:cp:hereAlgPf}, there is a Hilbert space $\mc{K}$, an isometry $V\colon \mathcal{H}\to \mathcal{K}$ and a $*$-representation $\rho\colon A\rightarrow \mc{B}(\mc{K})$ such that $\pi(x)=V^* \rho(x) V$ for all $x \in eAe$. A straightforward calculation reveals
\[ \tilde{\pi}\rtimes \lambda^p((\mathbb{E}\rtimes_1 G )(f)) =
(V^* \rho(e)\otimes 1)\,(\tilde{\rho}\rtimes \lambda^p) (f) \,  (\rho(e)V\otimes 1),
	\]
	hence
	\begin{align*}
		\| \tilde{\pi}\rtimes \lambda^p( (\mathbb{E}\rtimes_1 G)(f))\| \leq \| \tilde{\rho}\rtimes \lambda^p (f)\|
	\end{align*}
	for all $f\in \ell^1(G;A)$. This implies that $\|(\mathbb{E}\rtimes_1 G)(f) \|_{\pi,p} \leq \| f\|_{\rho,p}$ for all $f\in \ell^1(G;A)$. Consequently $\| (\mathbb{E}\rtimes_1 G)(f)\|_{\PF{p}{G}{eAe}}\leq \| f \|_{\PF{p}{G}{A}}$ and $\mathbb{E}\rtimes_1 G$ extends to a contraction $\PF{p}{G}{A}\rightarrow \PF{p}{G}{eAe}$.
\end{proof}

We are now ready for the theorem mentioned above.

\begin{thm}\label{thm:projectionprop}
Let $p\in (1,\infty)$. The functor $\rtimes_{\textup{PF}_p} G$ has the projection property, hence it is a correspondence functor.
\end{thm}

\begin{proof}
	Let $A$ be a $G$-C*-algebra and $e\in \mathcal{M}(A)$ a $G$-invariant projection. We need to show that the canonical inclusion $\iota\colon eAe \to A$ induces a injective $*$-homomorphism $\iota\rtimes_{\textup{PF}_p} G \colon eAe\rtimes_{\textup{PF}_p}G \to A\rtimes_{\textup{PF}_p}G$.

By Proposition \ref{cp:cp:hereAlgPf}, we can view $\textup{PF}_p^*(G;eAe)$ as
	a Banach $*$-subalgebra of $\textup{PF}_p^*(G;A)$. Let
$N_1\colon \textup{PF}_p^*(G,eAe)\to[0,\infty)$ and $N_2\colon \textup{PF}_p^*(G,A)\to [0,\infty)$ denote the restriction of the C*-norm of $C^*(\textup{PF}_p^*(G;eAe))$ and $C^*(\textup{PF}_p^*(G;A))$, respectively. Since any $*$-representation of $\PF{p}{G}{A}$ can be restricted to $\PF{p}{G}{eAe}$, we see that $N_1 \geq N_2\vert_{\textup{PF}_p^*(G;eAe)}$. Thus it remains	to show that $N_1 \leq N_2\vert_{\textup{PF}_p^*(G,eAe)}$.
	
Take $0\neq x\in \textup{PF}_p^*(G;eAe)$. There is
	a $*$-representation $\Pi\colon \textup{PF}_p^*(G;eAe)\to \mathcal{B}(\mathcal{H})$ on a Hilbert space $\mathcal{H}$, with the unit cyclic vector $\xi \in \mathcal{H}$ such
	that $N_1(x^*x) = \la \xi ,  \Pi(x^*x)\xi\ra $. Let $\varphi = \langle \xi, \Pi(\cdot)\xi \rangle$
	denote the corresponding vector state. We claim that
	\begin{align*}
		\tilde{\varphi} = \varphi \circ (\mathbb{E}\cpf) \colon \PF{p}{G}{A} \to \mathbb{C}
	\end{align*}
	is a positive linear functional, where $\mathbb{E}\colon A\ni a\mapsto e a e\in eAe$ is the canonical conditional expectation. Since $\rtimes_{u}G$ is a correspondence functor, $eAe\rtimes_{u}G$ is a hereditary sub-C*-algebra of $A\rtimes_{u}G$ and the $G$-equivariant
	(completely positive) map $\mathbb{E}\colon A \to eAe$ induces a completely positive map $$\mathbb{E}\rtimes_{u}G \colon A\rtimes_{u}G \to
	eAe\rtimes_{u}G.$$	
	Note that $\varphi':=\varphi \circ \iota_{eAe,p}$ is a positive linear functional on $\ell^1(G;eAe)$. Since
	$$C^*(\ell^1(G;eAe)) = eAe\rtimes_{u}G,$$ there is a unique continuous extension of $\varphi'$, $\varphi'' \colon eAe\rtimes_{u}G \to \mathbb{C}$
	such that we  obtain the following commutative diagram
	\begin{equation}
		\xymatrix{
			A\rtimes_{u}G \ar[r]^-{\mathbb{E}\rtimes_{u}G}& eAe\rtimes_{u}G \ar[rd]^-{\varphi''}\\
			\ell^1(G;A) \ar[u] \ar[r]^-{\mathbb{E}\rtimes_1 G} \ar[d]^-{\iota_{A,p}}& \ell^1(G;eAe) \ar[u]
			\ar[d]^-{\iota_{eAe,p}} \ar[r]^-{\varphi'}& \mathbb{C}\\
			\PF{p}{G}{A}\ar[r]^-{\mathbb{E}\cpf }& \PF{p}{G}{eAe} \ar[ru]^-{\varphi}.
		}
	\end{equation}
The linear map $\varphi'' \circ (\mathbb{E}\rtimes_{u} G)$ is positive, therefore
	$\varphi' \circ (\mathbb{E}\rtimes_1 G)$ is a positive functional. Since $\tilde{\varphi} \circ \iota_{A,p}=
	\varphi' \circ (\mathbb{E}\rtimes_1 G)$, and $\iota_{A,p}$ has dense image, the functional
	$\tilde{\varphi}$ is positive.
Consider the GNS representation for $\tilde{\varphi}$, as described on \cite[Page 198]{BonsallDuncan}. It gives us a Hilbert space $\mc{H}_{\tilde{\varphi}}$, $*$-representation $\Pi_{\tilde{\varphi}}\colon \PF{p}{G}{A}\rightarrow \mc{B}(\mc{H}_{{\tilde{\varphi}}})$ and a contractive linear map $\Lambda_{{\tilde{\varphi}}}\colon \PF{p}{G}{A}\rightarrow \mc{H}_{\tilde{\varphi}}$ satisfying $\Pi_{\tilde{\varphi}}(y)\Lambda_{\tilde{\varphi}}(y')=\Lambda_{\tilde{\varphi}}(yy')$, $\la \Lambda_{\tilde{\varphi}}(y) , \Lambda_{\tilde{\varphi}}(y')\ra = \tilde{\varphi}(y^*y')$. Let $(u_i)_{i\in I}$ be a contractive approximate unit in $\PF{p}{G}{A}$ (Proposition \ref{prop:approximate}). We obtain	
	\[\begin{split}
N_1(x)^2&=N_1(x^*x)=\la \xi , \Pi(x^*x)\xi\ra=
\varphi(x^*x)=\tilde{\varphi}(x^*x)=
\lim_{i\in I} \tilde{\varphi}(u_i^* x^*x u_i)\\
&=
\lim_{i\in I} \la \Lambda_{\tilde{\varphi}}(u_i) , \Pi_{\tilde{\varphi}}(x^*x) \Lambda_{\tilde{\varphi}}(u_i)\ra \le
\|\Pi_{\tilde{\varphi}}(x^*x)\| \le N_2(x^*x)=N_2(x)^2.
\end{split}\]	
%	 Hence, the GNS-construction $(\tilde{\mathcal{H}},\tilde{\Pi},\tilde{\xi})$ of $\tilde{\varphi}$ is an extension of $\Pi$, i.e.
%	$\tilde{\Pi}(x)\vert_{\tilde{\pi}(\PF{p}{G}{eAe})\tilde{\mathcal{H}}} = \Pi(x)$
%	for all $x \in \PF{p}{G}{eAe}$. Thus
%$$N_1(x) = \|\Pi(x)\| = \| \tilde{\Pi}(x)\vert_{\tilde{\Pi}(\PF{p}{G}{eAe})\tilde{\mathcal{H}}} \|\leq \| \tilde{\Pi}(x) \| \leq N_2(x).$$
\end{proof}

Recall the Brown-Guentner crossed product functor (\cite[Section 2.2]{BEW}).

\begin{cor}
Let $p \in (1, \infty)$ and let $G$ be non-amenable. Then the crossed product functors $\rtimes_{\textup{BG}_p}G$ and $\rtimes_{\textup{PF}_p}G$ are different.
\end{cor}
\begin{proof}
It is well-known that $\rtimes_{\textup{BG}_p}G$ is not a correspondence functor (see \cite[Remark 4.19]{BEW}).
\end{proof}

One can in fact also give simple examples where the resulting crossed products are different: for example if $A$ is equipped with a trivial action, then
$A \rtimes_{\textup{BG}_2} G = A \otimes_{\textup{max}} C^*_r(G)$, $A \rtimes_{\textup{PF}_2} G = A \otimes C^*_r(G)$. The interest in the statement above stems from the fact that we still do not know an example of a group $G$ for which $C^*(\textup{BG}_p(G)) \neq C^*(\textup{PF}^*_p(G))$.

\section{Non-triviality of the construction} \label{sec:examples}

In this final section, we shall discuss several examples in which our functor leads to genuinely exotic crossed products.

We begin with a straightforward observation: in the presence of an invariant state on $A$, to show that the crossed products $C^*(\textup{PF}^*_{p}(G;A))$ are exotic, one may simply use the corresponding fact for the algebras $C^*(\textup{PF}^*_p(G))$.

\begin{thm}\label{T:exotic C*-alg imply exotic cross product}
Suppose that $A$ is a unital $G$-C*-algebra which admits a $G$-invariant state. Assume further that $p, p' \in [2,\infty), p <p'$ are such that the canonical contractive map $q\colon C^*(\textup{PF}^*_{p'}(G)) \to C^*(\textup{PF}^*_{p}(G)) $ is not isometric. Then the corresponding map $q_A\colon C^*(\textup{PF}^*_{p'}(G;A)) \to C^*(\textup{PF}^*_{p}(G;A)) $ is also not isometric.
\end{thm}
\begin{proof}
This is an immediate consequence of Proposition \ref{invariantstate}, Theorem \ref{thm:projectionprop} and Theorem \ref{thm:coincide}.
\end{proof}

The above proposition shows that one can apply the results of \cite{SW2} (particularly, \cite[Corollary 4.3 and Corollary 5.12]{SW2}) to show that many classes of the crossed products we constructed are, in fact, exotic. We also have a similar fact for trivial actions (which naturally admit invariant states), as follows already from  Proposition \ref{prop:triv} (via Theorem \ref{thm:projectionprop} and Theorem \ref{thm:coincide}, as we can view  $C^*(\textup{PF}^*_p(G;A))$ as a C*-completion of $A \odot C^*(\textup{PF}^*_p(G))$. Note that we always have $A\subseteq C^*(\textup{PF}^*_p(G;A))$, as $\rtimes_{\textup{PF}_p} G$ is a crossed product functor (see the discussion in \cite[Section 2.1]{BEW}).

We would like to show that in some cases we can use the methods developed in \cite{SW2} also for actions which do not admit invariant states.
We first need a definition, introduced in \cite[Definition 4.1]{SW2}, which presents a modest strengthening of the Haagerup property for groups.

\begin{defn}\label{D:IHP}
Let $B(G)=C^*(G)^*$	be the Fourier-Stieltjes algebra of $G$. %, and let $B_{\ell^p}(G)=(B(G)\cap \ell^p(G))^{\sigma(B(G),C^*(G))}$.
Then $G$ has the \emph{integrable Haagerup property} if
	$$ B(G)=\overline{\bigcup_{1\leq p<\infty}(B(G)\cap \ell^p(G))}^{\,\sigma(B(G),C^*(G))},$$

	Equivalently, $G$ has the integrable Haagerup property if there exists a net of positive definite functions% (on $G$)
	$$\{\phi_i\}_{i \in I}\subseteq \bigcup_{1\leq p<\infty}\ell^p(G)$$
	so that $\phi_i\xrightarrow[ i \in I]{} 1$ uniformly on finite subsets of $G$.
%Again equivalently, there is a proper conditionally negative function $L$ on $G$ such that for all $\beta\in (0,1)$,
%$$\varphi_\beta:=\beta^L\in \bigcup_{1\leq p<\infty}\ell^p(G).$$
\end{defn}

We also recall a fact established in \cite[Lemma 3.2]{SW2}: $\ell^p$-integrable positive definite functions on $G$ give rise to bounded positive functionals on $C^*(\textup{PF}^*_p(G))$.

Before we formulate the next theorem, we should warn the reader that we shall use the notation of the type $C^*(\textup{PF}^*_p(G;A)) \neq A\rtimes_r G$ to say that the corresponding canonical map is not isometric (and not that the C*-algebras in question are not abstractly isomorphic, which seems to be much more difficult to establish).

\begin{thm} \label{thm:Haagerup}
  Suppose that $A$ is a $G$-C*-algebra such that
  $$A\rtimes_u G\neq A\rtimes_r G.$$ Then
there is $p\in (2,\infty)$ such that
  \begin{align*}
    A\rtimes_u G\neq C^*(\textup{PF}^*_p(G;A)).
  \end{align*}
If moreover $G$ has the integrable Haagerup property, then
  there is $p\in (2,\infty)$ such that
  \begin{align*}
 C^*(\textup{PF}^*_p(G;A)) \neq A\rtimes_r G.
  \end{align*}
\end{thm}

\begin{proof}
Let $f\in \ell^1(G;A)$. It follows from the interpolation relation in Proposition \ref{CP:cov:interpolationPseodo}, applied for $p_0=2$ and $p_1$ tending to infinity, that for every $p\in (2,\infty)$ and $\theta=2/p$, we have
  $$\|f\|_{C^*(\textup{PF}^*_p(G;A))}\leq \|f\|_{\textup{PF}^*_p(G;A)}\leq \|f\|^\theta_{A\rtimes_r G} \|f\|^{1-\theta}_{\ell^1(G;A)}.$$
  Hence, by Proposition \ref{prop:monotone},
  $$\|f\|_{A\rtimes_r G} \leq \inf_{p> 2} \|f\|_{C^*(\textup{PF}^*_{p}(G;A))}=\lim_{p\to 2^+} \|f\|_{C^*(\textup{PF}^*_{p}(G;A))}\leq \|f\|_{A\rtimes_r G}$$
  so that
  $$\|f\|_{A\rtimes_r G} = \inf_{p> 2} \|f\|_{C^*(\textup{PF}^*_{p}(G;A))}.$$
  Thus our hypothesis implies that for some $f\in \ell^1(G;A)$ and some $p\in (2,\infty)$, we must have
  \begin{align*}
   \|f\|_{A\rtimes_u G}\neq \|f\|_{C^*(\textup{PF}^*_{p}(G;A))}.
  \end{align*}

Suppose now that $G$ has the integrable Haagerup property and that for every $p\in [2,\infty)$ we have $$C^*(\textup{PF}^*_p(G;A))= A\rtimes_r G.$$
 By Definition \ref{D:IHP},  there exists a net of positive definite functions% (on $G$)
	$$\{\phi_i\}_{i \in I}\subseteq \bigcup_{1\leq p<\infty}\ell^p(G)$$
	such that $\phi_i\xrightarrow[ i \in I]{} 1$ uniformly on finite subsets of $G$.
Now fix $i\in I$ and choose $p\geq 2$ such that $\phi_i \in \ell^p(G)$. Suppose that $\widetilde{m_{\phi_i}}$ is the associated map defined in \eqref{usualSchur2}. By Proposition \ref{Schurrevisited} and the discussion before the theorem, $\widetilde{m_{\phi_i}}$  extends to a (completely) contractive map from $A\rtimes_r G=C^*(\textup{PF}^*_p(G;A))$ into $A\rtimes_u G$. Thus, for every $f\in \ell^1(G;A)$, we have
  \begin{align*}
    \|f\|_{A\rtimes_u G} &= \lim_{i\to \infty} \|\widetilde{m_{\phi_i}}(f)\|_{A\rtimes_u G}
\leq  \|f\|_{A\rtimes_r G}.
  \end{align*}
  This is a contradiction, which completes the proof.
\end{proof}

It would be interesting to see whether one could strengthen the statement of the previous theorem and show that, in some cases, $C^*(\textup{PF}^*_p(G;A))$ can produce an exotic crossed product without knowing whether there is an invariant state on $A$. We finish this section by presenting a generic way of producing many such examples. Our method relies on random walks and compact stationary spaces. First, we briefly recall the concepts and the notation we need.

Let $\mu$ be finitely supported probability measure on a countable group $G$. We \emph{always} assume that $\mu$ is {\it non-degenerate}, i.e.~its support generates $G$ as a semigroup. The \emph{Avez entropy} of $\mu$ in $G$, denoted by $h(G,\mu)$,  is defined as
\[
h(G,\mu):=\lim_{n\to \infty} \frac{-1}{n}\sum_{s\in G} \mu^{*n} (s) \, \log \mu^{*n} (s),
\]
where we use the convention that $\mu^{*n}(s)\log \mu^{*n}(s)=0$ whenever $s\notin \supp \, \mu^{*n}$. For a sub-additive function $L$ on $G$ with values in $\R_+$, we define the \emph{speed} of $\mu$ in $G$ with respect to $L$ to be
\[
\ell(G,\mu):=\lim_{n\to \infty} \frac{1}{n}\sum_{s\in G}\mu^{*n} (s) L(s).
\]
Let $X$ be a compact metrizable space such that $G$ acts continuously on $X$, and let $\mu$ be a finitely supported non-degenerate probability measure on $G$. The quasi-invariant probability measure $\xi$ on $X$ is said to be {\it stationary} if $\mu*\xi=\xi$. In this case, we say that $(X,\xi)$ is a compact $(G,\mu)$-space.
%standard probability space equipped with a measurable $G$-action. We assume the action to be non-singular (also called quasi-invariant), i.e.~for every $s\in G$, the measures $s\xi$ and $\xi$ are mutually absolutely continuous (where $(s\xi)(A)=\xi(s^{-1}A)$ for measurable subsets $A \subseteq X$). For $1 \leq p < \infty$,
The \emph{Koopman representation} of $G$ on $L^2(X,\xi)$ is the map
\begin{align*} %\label{Eq:Lp-Kooopman rep-defn}
    \pi_{X}\colon G\to \mc{B}(L^2(X,\xi))
\end{align*}
given by
\begin{align*} %\label{Eq:Lp-Kooopman rep-formula}
[\pi_{X}(s)f](x)=\left[\frac{d(s\xi)}{d\xi}(x)\right]^{1/2} f(s^{-1}x)
\end{align*}
for all $s \in G$, $f\in L^2(X,\xi)$ and $\xi$-a.e.~$x$. Here,
$\frac{d(s\xi)}{d\xi}$ denotes the Radon-Nikodym derivative of $s\xi$ with respect to $\xi$.
The \emph{Furstenberg entropy} of $(X,\xi)$ with respect to $(G,\mu)$
%, going back to \cite{FurstenbergNCRP} (see also \cite{NevoZimmer,Kaim-Vers 1}),
is defined as
\begin{align*} %\label{Eq:Furstenberg entropy}
  h_\mu(X,\xi)=- \sum_{s\in G} \mu(s) \int_X  \log \left[\frac{d(s^{-1}\xi)}{d\xi}(x)\right] d\xi(x).
\end{align*}
It is well known that $h_\mu(X,\xi)\geq 0$, with equality occuring if and only if $\xi$ is $G$-invariant. Also
\begin{align*}
h_{\mu^{*n}}(X,\xi)=nh_\mu(X,\xi)\qquad( n\in \N),
\end{align*}
and
%, it was shown by Kaimanovich and Vershik \cite{Kaim-Vers 1} that if $\mu$ has finite Shannon entropy, then
\begin{align} \label{Eq:maximality of Furstenberg boundary}
h(G,\mu)\geq h_\mu(X,\xi),
\end{align}
where $h(G,\mu)$ is the Avez entropy of $\mu$.
%(\Pi_{\mu},\nu_{\infty})$ is the Poisson boundary of $(G,\mu)$ and $(X,\xi)$ is an arbitrary $(G,\mu)$-stationary space. Moreover, for every $\mu$-boundary $(X,\xi)$ (see e.g.~\cite{Furs-Glas 1} for the terminology),
The equality in \eqref{Eq:maximality of Furstenberg boundary} occurs if and only if $(X,\xi)$ is a measure preserving extension of the Poisson boundary of $(G,\mu)$
%there is a measurable preserving $G$-equivariant map from $(X,\xi)$ onto the Poisson boundary of $(G,\mu)$
(see \cite{Furs-Glas 1}).
%$(\Pi_{\mu},\nu_{\infty})$.

We state the following result which is a consequence of \cite[Theorem 3.11]{ASdLSS2} (and also follows from the proof of  \cite[Theorem 4.7]{ASdLSS1}).  The proof is straightforward and we present it here for the sake of completeness.

\begin{prop}\label{P:Lyap Expo dominates entropy-inverse integrable weights}
Let $\mu$ be a finitely supported, non-degenerate probability measure on a countable group $G$, and let
$\omega\colon G \to (0,\infty)$ be such that $\omega^{-1} = 1/\omega \in \ell^1(G)$. Then
\begin{align*}
h(G,\mu) \leq \liminf_{n\to \infty} \frac{1}{n}\sum_{s\in G}\mu^{* n} (s) \log \omega (s).
\end{align*}
%Moreover, we can choose $\bar\om$ that satisfying above and also
%\begin{align}\label{Eq:Lyap Expo-algebraic weight}
%  \bar\om^{-1}*\bar\om^{-1} \leq d\, \bar\om^{-1},
%\end{align}
%where $d$ a positive constant. In this case, we have
%\begin{align*}
%h(G,\eta) = \lim_{n\to \infty} \frac{1}{n}\int_{G}\log \omega (s)d\eta^{*n}(s).
%\end{align*}

\end{prop}

\begin{proof}
Let $\eta:=C\omega^{-1}$, where $C=\|\omega^{-1}\|_1^{-1}$. Then $\eta$ is a probability measure on $G$, and so, by Gibbs' Inequality, the relative entropy
$$D(\mu^{*n}||\eta):=\sum_{s\in G} \mu^{*n}(s)[\log \mu^{*n}(s)-\log \eta(s)] \qquad (n\in \mathbb{N}),$$
is non-negative. Thus
\begin{align*}
-\sum_{s\in G} \mu^{*n}(s)\log \mu^{*n}(s) &\leq -\sum_{s\in G} \mu^{*n}(s)\log \eta(s)\\
&=-C+\sum_{s\in G}\mu^{* n} (s) \log \omega (s).
\end{align*}
The final result follows by dividing both sides of the above inequality by $n$ and taking the liminf.
\end{proof}

Before presenting one of the main results of this section, we briefly compare our approach with earlier work of the second author and his collaborators. In \cite{ASdLSS1}, a new method was developed to characterize measurable amenability of certain stationary actions, including those of free groups. The central idea was to use Furstenberg entropy to show that, whenever the associated Koopman representation is weakly contained in the left regular representation, the Avez and Furstenberg entropies coincide, yielding a characterization of (Zimmer) amenability (see also \cite{ASdLSS2} for extensions to non-discrete groups).

We adopt a similar philosophy by studying the consequences of the equality of the  $C^*(\textup{PF}^*_p(G;C(X)))$ with either the reduced or the full crossed product, where $(X,\xi)$ is a $(G,\mu)$-space. While our approach is partly inspired by \cite{ASdLSS1} and \cite{ASdLSS2}, the methods developed here apply in a broader setting and require arguments that do not appear in the articles mentioned above. We therefore provide complete proofs.  Although our main results are established for free groups, they extend to a much wider class of groups (see Remark \ref{R:nontrivial PF_p cross C* alg using entropy}). We use the convention regarding the `inequalities' of the type $C^*(\textup{PF}^*_p(G;A)) \neq A\rtimes_r G$ introduced before Theorem \ref{thm:Haagerup}.

\begin{prop}\label{P:disticts PF crossed product}
    Let $\F_k$ be the free group on $2 \le k <\infty$ generators, let $p\in \left[2,\infty\right)$, let $\mu$ be a finitely supported, non-degenerate probability measure on $\F_k$, and let $(X,\xi)$ be a compact $(\F_k,\mu)$-space. Then
    \begin{itemize}
   \item[(i)] if $h_\mu(X,\xi)<h(\F_k,\mu)$, then $$C(X)\rtimes_u \F_k \neq C(X)\rtimes_r \F_k;$$
      \item[(ii)] if $h_\mu(X,\xi)<\frac{2}{p} h(\F_k,\mu)$, then $$C^*(\textup{PF}^*_p(\F_k;C(X))) \neq C(X)\rtimes_u \F_k;$$
      \item[(iii)]  if $h_\mu(X,\xi)<h(\F_k,\mu)-\frac{2\log (2k-1)}{p}\ell(\F_k,\mu)$, % where $\ell(\F_k,\mu)=\lim_{n\rightarrow \infty}\dfrac{1}{n}\sum_{s\in F_n}\mu^{*n}(s)\log L(s)$,
     then $$C^*(\textup{PF}^*_p(\F_k;C(X))) \neq C(X)\rtimes_r \F_k.$$
     \end{itemize}
\end{prop}

\begin{proof}
%Let $\{s_1,\ldots,s_k\}$ be a standard generating set of $\F_k$, let $\mu$ be the simple random walk on $\F_k$, and let $(X,\xi)$ be a compact $(\F_k,\mu)$-space.
Let us first fix some terminology. Let $S:=\{s_1,\ldots,s_k\}$ be the standard generating set for $\F_k$, let $L$ be the word-length function on $\F_k$ coming from $S\cup S^{-1}$, and let $\omega_\alpha$ be a function on $\F_k$ defined by
$$\omega_\alpha(s)=(1+L(s))^\alpha \qquad (s\in \F_k,\alpha\geq 0).$$
By well-known results of Haagerup from \cite{Haag}, $L$ is a conditionally negative definite function such that, for $\alpha$ large enough, $\ell^2(\F_k,\omega_\alpha)\subseteq C^*_r(\F_k)$. Also, consider the covariant representation $(\rho,\pi_X)$ of $\ell^1(\F_k)$ and $C(X)$, where $\pi_X$ is the Koopman representation of $\F_k$ and $\rho$ is the multiplication representation of $C(X)$ on $L^2(X,\xi)$, respectively, and let $\mathcal{A}$ be the C*-algebra generated in
$\mathcal{B}(L^2(X,\xi))$ by the image of $\ell^1(\F_k;C(X))$ with respect to the integrated form of $(\rho,\pi_X)$.

We are ready to prove the statements in the proposition.

\smallskip

(i) It follows from \cite[Corollary to Theorem 3.5]{Nevo1} or \cite[Theorem 9.2]{nevosageev} that the action of $\F_k$ on $(X,\xi)$ is not measurably amenable. Hence, it cannot also be topologically amenable, so that $C(X)\rtimes_u \F_k \neq C(X)\rtimes_r \F_k$ by \cite{Matsumura}, as free groups are exact.

(ii) Suppose that
$$C^*(\textup{PF}^*_p(\F_k;C(X))) = C(X)\rtimes_u \F_k.$$ Then, by Theorem \ref{thm:coincide} and Proposition \ref{P:vector value extension of homomorphism} and the fact that $\mathcal{A}$ is a quotient of $C(X)\rtimes_u \F_k$, the identity mapping on $\ell^1(G)$
%    \begin{align*}
%        f\mapsto f\otimes 1_X \ , \ \ell^1(G)\to A
%    \end{align*}
extends to a $*$-homomorphism from $C^*(\textup{PF}^*_p(\F_k))$ onto $C^*_{\pi_X}(\F_k)$, the C*-algebra in $\mathcal{B}(L^2(X,\xi))$ generated by the image of $\pi_X$. In particular, for $\alpha>0$ large enough, the coefficient function
$$\Xi_X(\cdot):=\la 1_X, \pi_X(\cdot) 1_X\ra$$ on $\F_k$, also known as the Harish-Chandra $\Xi$-function, belongs to
$$C^*_{\pi_X}(\F_k)^*\subseteq C^*(\textup{PF}^*_p(\F_k))^* \subseteq \ell^p(\F_k,\om_\alpha^{-1}),$$ where the last inclusion follows from \cite[Theorem 5.11]{SW2}. In other words,
\begin{align*}
    \Xi_X^p \om_\alpha^{-1}\in \ell^1(\F_k).
\end{align*}
It follows from Proposition \ref{P:Lyap Expo dominates entropy-inverse integrable weights} that
%\cite{ASdLSS2} (see also the proof of \cite[Theorem 5.2]{ASdLSS1}) that
\begin{align*}%\label{Eq:Lyap expo-weight coming from Harich Chandra function-inverse square integrable-Discretization}
h(\F_k,\mu)& \leq \liminf_{n\rightarrow \infty}\dfrac{1}{n}\sum_{s\in \F_k}\mu^{*n}(s)\log \left[\Xi_X(s)^{-p}\omega_\alpha(s)\right] \\
& = \liminf_{n\rightarrow \infty}\dfrac{1}{n}\sum_{s\in \F_k}\mu^{*n}(s)\left[-p\log \Xi_X(s)+\alpha \log (1+L(s))\right] \ \ \ (*)
\\
& = p\liminf_{n\rightarrow \infty}\dfrac{-1}{n}\sum_{s\in \F_k}\mu^{*n}(s)\log \Xi_X(s) \\
& \leq p\liminf_{n\rightarrow \infty}\dfrac{-1}{n}\sum_{s\in \F_k} \mu^{*n}(s)\int_X  \log\Bigl[ \frac{d(s^{-1}  \xi)}{d\xi}(x)^{1/2}\Bigr] d\xi(x) \\
&=\frac{p}{2}\liminf_{n\rightarrow \infty}\dfrac{1}{n} h_{\mu^{*n}}(X,\xi) \\
& = \frac{p}{2} h_\mu(X,\xi),
\end{align*}
where in ($*$) we used the fact that, by \cite[Theorem 3.7]{ASdLSS1},
$$ \lim_{n\rightarrow \infty}\dfrac{1}{n}\sum_{s\in \F_k} \mu^{*n}(s) \log (1+L(s))=0.$$
%where $h_\mu(\Pi_{\mu},\nu_{\infty})$ is the Furstenberg entropy of the $\mu$-random walk on its Poisson boundary.
Hence
\begin{align}\label{Eq:1}
   h_\mu(X,\xi)<\frac{2}{p} h(\F_k,\mu)   \Longrightarrow C^*(\textup{PF}^*_p(\F_k;C(X))) \neq C(X)\rtimes_u \F_k.
\end{align}
(iii) Suppose that $$C^*(\textup{PF}^*_p(\F_k;C(X))) = C(X)\rtimes_r \F_k.$$ By \cite{Haag}, $L$ is conditionally negative definite on $\F_k$ so that, by Sch\"onberg's theorem, for every $\beta\in (0,1)$, $\varphi_\beta=\beta^L$ is a positive definite function on $\F_k$. Also, it is easy to verify that $\varphi_\beta\in \ell^p(\F_k)$ if $\beta <(2k-1)^{-1/p}$. Suppose that $0<\beta <(2k-1)^{-1/p}$ and $\widetilde{m_{\varphi_\beta}}$ is the associated map defined in \eqref{usualSchur2}. By Proposition \ref{Schurrevisited} and \cite[Theorem 5.11]{SW2}, $\widetilde{m_{\varphi_\beta}}$  extends to a (completely) contractive map from $C(X)\rtimes_r \F_k=C^*(\textup{PF}^*_p(\F_k;C(X)))$ to $C(X)\rtimes_u \F_k$. This, together with the fact that $\mathcal{A}$ is a quotient of $C(X)\rtimes_u \F_k$, implies that we have a completely contractive mapping
\begin{align*}
    C^*_r(\F_k) \to C^*_{\pi_X}(\F_k),\quad f\mapsto f\varphi_\beta.
\end{align*}
In particular, using again \cite[Theorem 5.11]{SW2}, for $\alpha>0$ large enough we have
\begin{align*}
    \Xi_X \varphi_\beta \in \ell^2(\F_k,\om_\alpha^{-1}).
\end{align*}

%\begin{align*}
%    \Xi_X \varphi_\beta \om_\alpha^{-1}\in \ell^2(\F_k).
%\end{align*}

Similar to part (ii), it again follows from Proposition \ref{P:Lyap Expo dominates entropy-inverse integrable weights} and \cite[Theorem 3.7]{ASdLSS1} that
%\cite{ASdLSS2} (see also the proof of \cite[Theorem 5.2]{ASdLSS1}) that
\begin{align*}%\label{Eq:Lyap expo-weight coming from Harich Chandra function-inverse square integrable-Discretization}
h(\F_k,\mu)& \leq 2\liminf_{n\rightarrow \infty}\dfrac{-1}{n}\sum_{s\in \F_k}\mu^{*n}(s) \log \left(\Xi_X(s) \beta^{L(s)}\right) \\
& = 2\liminf_{n\rightarrow \infty}\left[\dfrac{-1}{n}\sum_{s\in \F_k}\mu^{*n}(s) \log \Xi_X(s) +\dfrac{-\log \beta}{n}\sum_{s\in \F_k}\mu^{*n}(s) L(s)\right] \\
&= 2\liminf_{n\rightarrow \infty}\left[\dfrac{-1}{n}\sum_{s\in \F_k}\mu^{*n}(s)\log \Xi_X(s)\right]- 2\ell(\F_k,\mu) \log \beta \\
&\leq h_\mu(X,\xi)- 2\ell(\F_k,\mu) \log \beta,
\end{align*}
where $\ell(\F_k,\mu)$ is  the speed of the random walk generated by $\mu$ in $\F_k$. By letting $\beta \to (2k-1)^{-1/p}$, we obtain
\begin{align}%\label{Eq:Lyap expo-weight coming from Harich Chandra function-inverse square integrable-Discretization}
h(\F_k,\mu)\leq h_\mu(X,\xi)+ \frac{2\ell(\F_k,\mu) \log (2k-1)}{p}.
\end{align}
%However, it is known that for simple symmetric random walk on $\F_k$,
%$$h(\F_k,\mu)=\ell(\F_k,\mu) \log (2k-1)=\frac{k-1}{k} \log (2k-1).$$
%Hence we must have
%\begin{align*}
%   \left(1-\frac{2}{p}\right) h(\F_k,\mu) \leq h_\mu(X,\xi).
%\end{align*}
Thus, we can state that
\begin{align}\label{Eq:2}
   h_\mu(X,\xi)<  h(\F_k,\mu)-\frac{2\log (2k-1)}{p}\ell(\F_k,\mu) \Longrightarrow C^*(\textup{PF}^*_p(\F_k;C(X))) \neq C(X)\rtimes_r \F_k.
\end{align}
%(iii) This is an immediate consequence of the fact that (i) implies (ii) if $p>4$.
%, then we can choose $(X,\nu)$ such that $h_\mu(X,\nu)$ satisfies both the relations \eqref{Eq:1} and \eqref{Eq:2}.
%On the other hand, if $C^*(C(X)\cpff) = C(X)\rtimes_r F_n$, then
\end{proof}

We can now exhibit further examples where the functors we introduced in the last section lead to exotic crossed products.

\begin{thm}\label{T:nontrivial PF_p cross C* alg using entropy}
    Let $\F_k$ be the free group on $2\le k < \infty$ generators, and let $\mu$ be a finitely supported, non-degenerate probability measure on $\F_k$. Then there is $p_0\geq 2$ so that for every $p\geq p_0$, there is a compact $(\F_k,\mu)$-space $(X,\xi)$ with $h_\mu(X,\xi)>0$ and
    \begin{align*}
        C(X)\rtimes_r \F_k\neq C^*(\textup{PF}^*_p(\F_k;C(X))) \neq C(X)\rtimes_u \F_k.
    \end{align*}
   % Moreover, $p$ can be chosen to be arbitrary large.
\end{thm}

\begin{proof}
It is shown in \cite{HY} (see also \cite{B1}, \cite{BZ1}) that for every $\epsilon \in (0,h(\F_k,\mu))$, there is a compact $(G,\mu)$-space $(X,\xi)$ such that $h_\mu(X,\xi)=\epsilon$. Hence if $p_0\geq 2$ is chosen such that $h(\F_k,\mu)-\frac{2\log (2k-1)}{p_0}\ell(\F_k,\mu)>0$, then for every $p\geq p_0$, one can find a compact $(G,\mu)$-space $(X,\xi)$ such that $h_\mu(X,\xi)$ is positive and satisfies both the conditions of Proposition \ref{P:disticts PF crossed product} (ii) and (iii). This completes the proof.
\end{proof}

\begin{rem}\label{R:nontrivial PF_p cross C* alg using entropy}
    The  preceding theorem can be generalized, with a fairly analogous proof, to a much larger class of groups. To see this, let $G$ be a finitely generated, nonamenable group. Suppose further that there is a proper conditionally negative length function $L$ on $G$ such that $G$ has both rapid decay and the integrable Haagerup property with respect to $L$. This class of groups includes free groups, but also many more examples (see \cite[Examples 5.13, 5.15 and Remark 5.14]{SW2}). Under these assumptions for $G$ and $L$, Proposition \ref{P:disticts PF crossed product} can be generalized, with the same proof, to $G$ and $\mu$, for any finitely supported non-degenerate probability measure $\mu$ on $G$. Hence Theorem \ref{T:nontrivial PF_p cross C* alg using entropy} holds as well, provided that one could find compact $(G,\mu)$-spaces with arbitrary small positive entropy. However, the latter fact is true, as it is a consequence of the failure of the entropy gap for $(G,\mu)$. Indeed, it follows from \cite{BHO} and \cite{Nevo1} that the existence of an entropy gap is equivalent to the group having Property (T). However, a nonamenable group with Haagerup property does not have Property (T) so that $(G,\mu)$ does not have an entropy gap. That is, for every $\delta>0$, there is a compact $(G,\mu)$-space $(X,\xi)$ such that $h_\mu(X,\xi)\in (0,\delta)$. Therefore, Theorem \ref{T:nontrivial PF_p cross C* alg using entropy} holds for $(G,\mu)$.
\end{rem}

Finally, we should say that it is not clear to us whether the actions discussed in Theorem \ref{T:nontrivial PF_p cross C* alg using entropy} and the last remark admit invariant states or not. In general it appears difficult to construct natural explicit examples of non-amenable actions with no invariant states (abstractly one may consider the action  of a non-exact $G$ on its Stone-\v{C}ech compactification). In the final theorem, we present one such class of actions; we do not know, however, if our functors lead to exotic crossed products also in this case.

\smallskip

Let $H$ be a connected semisimple Lie group with finite center and real rank at least 2, and let $P$ be the minimal parabolic subgroup of $H$ so that $H=PK$, where $K$ is the maximal compact subgroup of $H$. Then $H/P$ is the Furstenberg (Poisson) boundary of $H$, and admits a unique $K$-bi invariant probability measure $\nu_P$. More generally, for $P\subseteq Q\subseteq H$ with $Q$ being a parabolic subgroup of $H$, there is a unique $K$-invariant probability measure $\nu_Q$ on $H/Q$ (\cite[Theorem 2.3]{Fur63}, see also \cite[Proposition 1.3]{NevoZimmer1999}).
%{\color{red} Some references? Zimmer's book? Or some work of Furstenberg?}

\begin{thm}
Let $H$, $P$, and $Q$ be as above with $Q$ a proper subgroup of $H$ and $G$ a lattice in $H$. Then  $C(H/Q)\rtimes_u G\neq C(H/Q)\rtimes_r G$ and there is no $G$-invariant state on $C(H/Q)$.
\end{thm}

\begin{proof}
	It follows from the Furstenberg work in \cite[Theorem 3]{Fur67} and \cite{Fur71} that there is a probability measure $\mu$ on $G$ with finite first moment w.r.t the Riemannian metric on $H$ such that $\nu_P$ is the unique $\mu$-stationary measure on $H/P$ and $(H/P,\nu_P)$ is the Furstenberg boundary of $(G,\mu)$. Moreover, $\nu_Q$ is the \emph{unique} $\mu$-stationary measure on $H/Q$ and $(H/Q,\nu_Q)$ is a $\mu$-boundary (in fact, every $\mu$-boundary is of this form). This, in particular, implies that $C(H/Q)$ has no $G$-invariant probability measure as every such measure must also be $\mu$-stationary, thus equal to $\nu_Q$, which is not possible ($\nu_Q$ is not $H$-invariant, by \cite[Theorem A]{NevoZimmer1999}, so that it cannot be $G$-invariant as $G$ is a lattice in $H$). Furthermore, since $(H/Q,\nu_Q)$ is a proper $\mu$-boundary, % so that its Furstenberg boundary is strictly smaller than by
    \cite[Theorem 9.2]{nevosageev} yields that $(H/Q,\nu_Q)$ is not measurably amenable (as a $G$-space). Hence the action of $G$ on $H/Q$ is not topologically amenable, so that  $C(H/Q)\rtimes_u G\neq C(H/Q)\rtimes_r G$.
	
%	Now, let $\sigma_Q:\Gamma\to B(L^2(H/Q,\nu_Q))$ be the corresponding Koopman representation and $\pi: C(H/Q) \to B(L^2(H/Q,\nu_Q))$ be the multiplication representation. If $C(H/Q)\rtimes_u \Gamma\cong C(H/Q)\rtimes_r \Gamma$, then for the covariant representation $(L^2(H/Q,\nu_Q,\pi,\sigma_Q)$, we would have%
	%\begin{align*}
	%	\|\sigma_Q(f)\|=\|(\pi\rtimes \sigma_Q)(1_{H/Q}\otimes f)\|\leq \|(\pi\rtimes \lambda_\Gamma)(1_{H/Q}\otimes f)\|=\|\lambda_\Gamma(f)\|.
%	\end{align*}
%	for all $f\in \ell^1(G)$. This in turn implies that $\sigma_Q$ is weakly contained in $\lambda_\Gamma$. But by \cite{ASdLSS2} we would then have that $h_\mu(H/P,\nu_P)=h_\mu(H/Q,\nu_Q)$, where $h_\mu(\cdot,\cdot)$ is the Furstenberg entropy. Then again by a result of Furstenberg, we would have that $(H/P,\nu_P)=(H/Q,\nu_Q)$. The last fact is not possible, as $(H/Q,\nu_Q)$ is a proper $\mu$-boundary.
%	 Hence $C(H/Q)\rtimes_u \Gamma\ncong C(H/Q)\rtimes_r \Gamma$.
\end{proof}

\section*{Acknowledgments}
This material is partially based upon work supported by the Swedish Research Council under grant no. 2021-06594 while ES and AS  were in residence at Institute Mittag-Leffler in Djursholm, Sweden during the programme `Operator Algebras and Quantum Information' in March 2026.

JK was partially supported by FWO grant 1246624N. AS was partially supported by the National Science Center Grant OPUS-29 UMO-
2025/57/B/ST1/00057. ES was partially supported by NSERC Discovery Grant RGPIN-2025-04833. ES would like to thank the Wenner-Grenn Foundation (the project GFOh2024-0025) which supported his sabbatical visit to Chalmers University of Technology, 2025-2026.

The mathematical work in this article predates the recent rapid advances in AI-assisted
mathematics. No AI tools were used to conduct the research or obtain mathematical
results of this manuscript.

%\section{List of outstanding questions}

%Working list of questions:
%\begin{itemize}
%\item[--] check whether the independence of the norm $\|\cdot\|_{\pi,p} $ on the (faithful) representation $\pi$ of $A$ holds beyond $A= \C$ -- this seems difficult!
%\item[--] where do we put Jacek's result -- in the appendix? Or inside the main text, explaining that this is due to him?
%\item[--] how does this functor compare to Brown-Guentner crossed product functors? I.e.- when they coincide, beyond the amenable case?
%\item[--] can we show the continuity result for the $C^*$-version of the $\ell^p$-regular norm? In some special cases?

%\end{itemize}

\end{document}